\documentclass[11pt]{article}
\usepackage{fullpage}
\usepackage{amsmath,amssymb,amsthm,amsfonts}
\usepackage[english]{babel}
\usepackage[left=3cm,right=3cm,top=3cm,bottom=4cm]{geometry}

\usepackage{tikz,tikz-cd}
\usetikzlibrary{positioning}
\usepackage{braket,mathtools,stmaryrd}

\SetSymbolFont{stmry}{bold}{U}{stmry}{m}{n}
\usepackage{lmodern,bm,bbm,mathrsfs,manfnt,MnSymbol}
\usepackage{enumitem}
\setlist{noitemsep}

\usepackage[textsize=tiny,colorinlistoftodos]{todonotes}

\usepackage[backref=page, colorlinks=true, linkcolor=blue, citecolor=red, menucolor=green]{hyperref}
\usepackage{setspace}

\usepackage[bottom]{footmisc}

\theoremstyle{plain}

\newtheorem*{proposition}{Proposition}
\newtheorem*{lemma}{Lemma}

\newtheorem*{conjecture}{Conjecture}
\theoremstyle{definition}
\newtheorem*{definition}{Definition}
\newtheorem*{example}{Example}
\theoremstyle{remark}
\newtheorem*{remark}{Remark}

\numberwithin{equation}{section}

\renewcommand{\bar}[1]{\overline{#1}}
\renewcommand{\hat}[1]{\widehat{#1}}

\renewcommand{\tilde}[1]{\widetilde{#1}}
\renewcommand{\vec}[1]{\bm{#1}}
\newcommand{\bC}{\mathbb{C}}
\newcommand{\bE}{\mathbb{E}}
\newcommand{\bP}{\mathbb{P}}
\newcommand{\bZ}{\mathbb{Z}}
\newcommand{\bk}{\mathbbm{k}}
\newcommand{\cF}{\mathcal{F}}
\newcommand{\cI}{\mathcal{I}}
\newcommand{\cK}{\mathcal{K}}

\newcommand{\cN}{\mathcal{N}}
\newcommand{\cO}{\mathcal{O}}
\newcommand{\cQ}{\mathcal{Q}}
\newcommand{\cT}{\mathcal{T}}
\newcommand{\cV}{\mathcal{V}}
\newcommand{\scF}{\mathscr{F}}
\newcommand{\scI}{\mathscr{I}}
\newcommand{\scM}{\mathscr{M}}
\newcommand{\scN}{\mathscr{N}}
\newcommand{\sA}{\mathsf{A}}
\newcommand{\se}{\mathsf{e}}
\newcommand{\sC}{\mathsf{C}}
\newcommand{\sE}{\mathsf{E}}
\newcommand{\sF}{\mathsf{F}}
\newcommand{\sR}{\mathsf{R}}
\newcommand{\sS}{\mathsf{S}}
\newcommand{\sT}{\mathsf{T}}
\newcommand{\sV}{\mathsf{V}}
\newcommand{\CY}{\mathrm{CY}}
\newcommand{\der}{\mathrm{der}}
\newcommand{\fin}{\mathrm{fin}}
\newcommand{\loc}{\mathrm{loc}}
\newcommand{\pt}{\mathrm{pt}}

\newcommand{\vir}{\mathrm{vir}}
\DeclareMathOperator{\Bl}{Bl}
\DeclareMathOperator{\BS}{BS}
\DeclareMathOperator{\ch}{ch}
\DeclareMathOperator{\cochar}{cochar}
\DeclareMathOperator{\coker}{coker}
\DeclareMathOperator{\DT}{DT}
\DeclareMathOperator{\cExt}{\mathcal{E}{\it xt}}
\DeclareMathOperator{\Ext}{Ext}
\DeclareMathOperator{\Frac}{Frac}
\DeclareMathOperator{\Free}{Free}
\DeclareMathOperator{\GL}{GL}
\DeclareMathOperator{\Gr}{Gr}
\DeclareMathOperator{\Hilb}{Hilb}
\DeclareMathOperator{\cHom}{\mathcal{H}{\it om}}
\DeclareMathOperator{\Hom}{Hom}
\DeclareMathOperator{\nlegs}{\#legs}
\DeclareMathOperator{\NP}{NP}
\DeclareMathOperator{\PT}{PT}
\DeclareMathOperator{\QM}{QM}
\DeclareMathOperator{\Quot}{Quot}
\DeclareMathOperator{\rank}{rank}
\DeclareMathOperator{\SO}{SO}

\DeclareMathOperator{\tr}{tr}
\DeclarePairedDelimiter{\inner}{\langle}{\rangle}

\title{The $4$-fold Pandharipande--Thomas vertex}
\author{Henry Liu}
\date{\today}

\begin{document}

\maketitle

\begin{abstract}
  We give a conjectural but full and explicit description of the
  (K-theoretic) equivariant vertex for Pandharipande--Thomas stable
  pairs on toric Calabi--Yau $4$-folds, by identifying torus-fixed
  loci as certain quiver Grassmannians and prescribing a canonical
  half of the tangent-obstruction theory. For any number of
  non-trivial legs, the DT/PT vertex correspondence can then be
  verified by computer in low degrees.
\end{abstract}

\begin{spacing}{1.15}
  \tableofcontents
\end{spacing}

\section{Introduction}

\subsection{Background and setup}

\subsubsection{}

Let $X$ be a non-singular quasi-projective toric Calabi--Yau $4$-fold
over $\bC$, and let $\PT_{\beta,n}(X)$ be the moduli space of
Pandharipande--Thomas (PT) stable pairs \cite{Pandharipande2009}
\begin{equation} \label{eq:pair}
  I = [\cO_X \xrightarrow{s} \cF]
\end{equation}
with $\ch(\cF) = (0, 0, 0, \beta, n) \in A^*(X)$. PT stability means
that $\cF$ is a pure $1$-dimensional sheaf and $\dim \coker s = 0$.
Denote the kernel/cokernel exact sequence by
\[ 0 \to \cI_C \to \cO_X \xrightarrow{s} \cF \to \cQ \to 0, \]
where $\cI_C$ is the ideal sheaf of the Cohen--Macaulay support curve
$C$ associated to the stable pair. We will always assume $C$ is
non-empty.

On the other hand, let $\DT_{\beta,n}(X)$ be the Donaldson--Thomas
(DT) moduli space \cite{Thomas2000} of ideal sheaves of curves
$Z \subset X$, viewed as ``DT-stable'' pairs \eqref{eq:pair} where
$\cF = \cO_Z$ and $s$ is the natural surjection. Note that
$\DT_{0,d}(X) = \Hilb(X, d)$ is the Hilbert scheme of $d$ points. PT
stable pairs originated as a more economical stability chamber in
which to study $\DT_{\beta,n}(X)$.

\subsubsection{}

In this paper we study enumerative invariants of $\DT_{\beta,n}(X)$
and $\PT_{\beta,n}(X)$ in equivariant K-theory, though all results
hold equally well in equivariant cohomology. Such K-theoretic
invariants require the following Oh--Thomas setup \cite{Oh2023}. For a
quasi-projective scheme $M$ with action of a torus $\sT$, let
$K_\sT(M)$ (resp. $K_\sT^\circ(M)$) denote the K-group of
$\sT$-equivariant coherent sheaves (resp. vector bundles) on $M$, {\it
  with base ring $\bZ[1/2]$ instead of $\bZ$}. They are modules for
$\bk_\sT \coloneqq K_\sT(\pt)$. A subscript $\loc$ denotes base change
to $\Frac \bk_{\tilde\sT}$ {\it where $\tilde\sT$ is a double cover of
  $\sT$}. Finally,
$\se_\sT(-) \coloneqq \sum_k (-1)^k \wedge^k(-)^\vee$ denotes
K-theoretic Euler class.

More specifically, for us $\sT = (\bC^\times)^4$ is the dense open
torus of $X$, with weights/coordinates denoted $t_1, t_2, t_3, t_4$.
Let $\kappa \coloneqq t_1t_2t_3t_4$ be the weight of the canonical
bundle $\cK_X \cong \kappa \otimes \cO_X$. Then there is a splitting
$\sT = \sA \times \bC_\kappa^\times$ where
$\sA \coloneqq \{\kappa = 1\}$ is the Calabi--Yau sub-torus.

\subsubsection{}

Let $M$ denote either $\DT_{\beta,n}(X)$ or $\PT_{\beta,n}(X)$. If
$\scI = [\cO_{M \times X} \to \scF]$ is the universal stable pair, and
$\pi_M\colon M \times X \to M$ is the projection, then
\begin{equation} \label{eq:obstruction-theory}
  \bE^\bullet \coloneqq R\cHom_{\pi_M}(\scI, \scI)_0^\vee[-1] \in D^b\mathscr{C}\!{\it oh}_\sT(M)
\end{equation}
is a $3$-term obstruction theory \cite[Proposition 4.1]{Oh2023},
symmetric in the sense that
$(\bE^\bullet)^\vee = \kappa \otimes \bE^\bullet[-2]$ by Serre
duality. In particular, $\sA$-equivariantly, the middle term is an
$\mathrm{O}(2m, \bC)$-bundle, and becomes an $\SO(2m, \bC)$-bundle
after choosing orientation data \cite{Borisov2017}. Results of
\cite{Borisov2017,Oh2023,Cao2020a,Bojko2021} then provide a localized
Oh--Thomas virtual cycle \footnote{We suppress the dependence of
  Oh--Thomas cycles on orientation data as it is irrelevant for the
  main point of this note, and assume throughout that suitable
  orientation data exist such that our claims hold, e.g.
  \eqref{eq:partition-function-factorization} below.}
\begin{equation} \label{eq:oh-thomas-cycle}
  \hat\cO^\vir_M = \iota_*\frac{\hat\cO^\vir_{M^\sA}}{\sqrt{\se_\sA}(\cN^\vir)} \in K_\sA(M)_{\loc}.
\end{equation}
where $\sqrt{\se_\sA}$ is the square-root Euler class \cite[\S
  5]{Oh2023}, $\iota\colon M^\sA \to M$ is the $\sA$-fixed locus,
and $\cN^\vir$ is the virtual normal bundle of $\iota$.

To emphasize, \eqref{eq:oh-thomas-cycle} (and
\eqref{eq:square-root-euler-class}) is only valid $\sA$-equivariantly,
not $\sT$-equivariantly, because $\bE^\bullet$ must actually be
symmetric instead of only symmetric up to an equivariant character.

\subsubsection{}

The square-root Euler class $\sqrt{\se_\sA}(E)$ of an
$\sA$-equivariant $\SO(2m, \bC)$-bundle $E$ is such that
\begin{equation} \label{eq:square-root-euler-class}
  (\sqrt{\se_\sA}(E))^2 = (-1)^m \hat \se_\sA(E)
\end{equation}
is the symmetrized \footnote{Note that $\det(E)^{1/2}$ is defined in
  K-theory following \cite[Lemma 5.1, Prop. 7.12]{Oh2023}, and exists
  even if $\det(E)$ does not admit a square root as a bundle.}
K-theoretic Euler class
$\hat \se_\sA(E) \coloneqq \se_\sA(E) \otimes \det(E)^{-1/2}$ up to a
sign. The remaining sign ambiguity is fixed by the orientation data
and, unlike $\se_\sA(E)$, truly depends on $E$ instead of just its
K-theory class. In other words, there is a K-theoretic splitting
\[ E = E_{1/2} + E_{1/2}^\vee \in K_\sA(\cdots), \]
certainly not unique, such that
\[ \sqrt{\se_\sA}(E) = \hat \se_\sA(E_{1/2}), \]
and a different choice of splitting $E_{1/2}'$ may change
$\sqrt{\se_\sA}(E)$ by a sign \footnote{If (a piece of) $F$ is
  nilpotent, we make sense of such an expression by implicitly adding
  a trivial $\bC^\times_t$-action, with respect to which $F$ has
  weight one, and then setting $t=1$ afterward, like in
  \cite{Thomas2022}.}
\[ \hat \se_\sA(E_{1/2} - E_{1/2}') = \hat \se_\sA(F - F^\vee). \]
In our situation, finding a correct (and perhaps also geometrically
meaningful) half $\cN^\vir_{1/2}$ is therefore a very delicate matter,
in some sense the only non-formal aspect of DT or PT theory on $X$.
This difficulty with signs was already recognized in the foundational
paper \cite{Nekrasov2020} studying invariants of
$\Hilb(\bC^4) = \bigsqcup_n \DT_{0,n}$. Our main
Conjecture~\ref{conj:DT-PT-Ovir} prescribes such a half of
\eqref{eq:obstruction-theory}, in order to explicitly identify the
cycle \eqref{eq:oh-thomas-cycle}.

\subsubsection{}

Let $M \in \{\DT, \PT\}$. The simplest K-theoretic enumerative
invariants of $M$ are
\begin{equation} \label{eq:partition-function}
  Z_\beta^M(X; y) \coloneqq \sum_{n \in \bZ} Q^n \chi\left(M_{\beta,n}, \hat\cO^\vir \otimes \hat \se_\sA(y \otimes \cO_X^{[n]})\right) \in \bk_{\sA \times \bC^\times_y,\loc}((Q)),
\end{equation}
where $\cO_X^{[n]} \coloneqq R\pi_{M*}(\scF)$ is the Nekrasov
insertion and we treat the variable $y$ coupled to it as an
equivariant variable. A basic conjecture in this subject is the {\it
  DT/PT correspondence}
\begin{equation} \label{eq:DT-PT-partition-function-correspondence}
  Z_\beta^{\PT}(X;y) \stackrel{?}{=} \frac{Z_\beta^{\DT}(X;y)}{Z_0^{\DT}(X;y)},
\end{equation}
generalizing the DT/PT correspondence for $3$-folds \cite[\S
3.3]{Pandharipande2009}. As $X$ is toric, the $\sA$-equivariant
localization \eqref{eq:oh-thomas-cycle} and standard considerations
\cite[\S 4]{Maulik2006}\cite[\S 2.4]{Cao2020} imply
\begin{equation} \label{eq:partition-function-factorization}
  Z_\beta^M(X; y) = \sum_{\vec\pi} \prod_{e\colon \alpha\to\beta} E_{\vec\pi(e)}(\vec t_{\alpha\beta}; y) \prod_\alpha V_{(\vec\pi(e_{\alpha,i}))_{i=1}^4}^M(\vec t_\alpha; y)
\end{equation}
factorizes according to the toric $1$-skeleton of $X$ into
contributions of vertices $\alpha \in V(X)$, which are toric charts
$U_\alpha$ with toric coordinates of weights denoted $\vec t_\alpha$,
and edges $[\alpha \to \beta] \in E(X)$, which are non-empty
intersections $U_\alpha \cap U_\beta$ with toric coordinates of
weights denoted $\vec t_{\alpha\beta}$. Here $(e_{\alpha,i})_{i=1}^4$
are the four incident edges at each vertex $\alpha$, and the sum is
over {\it edge labelings} $\vec\pi\colon E(X) \to \Pi_{3d}$ by solid
partitions called the {\it profiles} of edges; see
\S\ref{sec:PT-full-fixed-loci} for details. Note that $E_\pi$ can be
arranged to be a purely combinatorial product independent of $M$. Then
\eqref{eq:DT-PT-partition-function-correspondence} may be refined into
the stronger conjectural {\it DT/PT vertex correspondence}
\begin{equation} \label{eq:DT-PT-vertex-correspondence}
  V^{\PT}_{\pi^1,\pi^2,\pi^3,\pi^4}(\vec t; y) \stackrel{?}{=} \frac{V^{\DT}_{\pi^1,\pi^2,\pi^3,\pi^4}(\vec t; y)}{V^{\PT}_{\emptyset,\emptyset,\emptyset,\emptyset}(\vec t; y)}.
\end{equation}

\subsubsection{}

The vertices $V^{\DT}$ and $V^{\PT}$ have a strong combinatorial
flavor, as they are sums of torus-localization contributions from a
fixed toric chart $\bC^4 \subset X$. For instance, the restrictions
$\cO_Z\big|_{\bC^4}$ of DT fixed points $[\cO_X \to \cO_Z]$ are in
bijection with {\it solid}, or {\it four-dimensional}, {\it partitions
  with legs} \cite[\S 2.2]{Nekrasov2020} \cite[Definition
2.3]{Cao2020}, where each box in the partition represents a
$\sT$-weight space of $\cO_Z\big|_{\bC^4}$. As such, $V^{\DT}$ is a
weighted generating function of solid partitions with legs,
generalizing the famous series $1/\prod_{n>0} (1-Q^n)$ and
$1/\prod_{n>0} (1-Q^n)^n$ counting integer (two-dimensional) and plane
(three-dimensional) partitions respectively.

Similarly, PT fixed points also correspond to certain types of box
configurations, possibly now with non-trivial moduli. However, the
description \cite{Pandharipande2009a,Jenne2021} of these PT
configurations for $3$-folds does not work analogously for $4$-folds
and a more sophisticated approach is necessary. A consequence of our
main conjecture is a method, as explicit as one could hope for, to
compute $V^{\DT}$ and $V^{\PT}$ term-by-term via an understanding of
the geometry of these fixed loci.

\subsection{Main results}

\subsubsection{}

Our first result is an identification of $\sA$-fixed loci (and
$\sT$-fixed loci) in PT moduli spaces as products of quiver
Grassmannians. Details are given in \S\ref{sec:fixed-loci}.

\begin{proposition} \label{prop:fixed-loci}
  Let $F_{\vec\pi} \subset \PT_\beta(X)^\sA$ be the components of
  stable pairs whose profiles along edges are given by the edge
  labeling $\vec\pi\colon E(X) \to \Pi_{3d}$. Then
  \[ F_{\vec \pi} \cong \prod_\alpha \Gr(\bar M^\CY_\alpha) \]
  where $\bar M^\CY_\alpha$ is the quiver of \S\ref{sec:PT-CY-quiver}
  for the four legs $(\vec \pi(e_{\alpha,i}))_{i=1}^4$ in the toric
  chart $U_\alpha \subset X$. Each $\Gr(\bar M^\CY_\alpha)$ carries a
  perfect obstruction theory.
\end{proposition}

There is a similar description of $\sT$-fixed PT loci in terms of a
related quiver $\bar M_\alpha$ (\S\ref{sec:PT-full-quiver}). In stark
contrast to the $3$-fold case \cite[Theorem 1]{Pandharipande2009a}, we
can give examples of arbitrarily singular and non-toric $\sT$-fixed PT
components.

Crucially, using the perfect obstruction theory on quiver
Grassmannians, $F_{\vec\pi}$, and therefore each
$\PT_{\beta,n}(X)^\sA$, has a well-defined symmetrized {\it
  Behrend--Fantechi}-style virtual cycle, which we denote
\begin{equation} \label{eq:PT-fixed-locus-BF-cycle}
  \hat\cO^{\vir,\mathrm{BF}} \in K_\sA(\PT_{\beta,n}(X)^\sA)
\end{equation}
to distinguish it from the already-existing Oh--Thomas virtual cycle
$\hat\cO^\vir$. For notational consistency, define
\eqref{eq:PT-fixed-locus-BF-cycle} for $\DT_{\beta,n}(X)^\sA$ as well,
as the ordinary structure sheaf.

\subsubsection{}

Following \S\ref{sec:canonical-half}, the $\sA$-fixed locus $M^\sA$
admits a further $\bC^\times$-action (coming from
$\bC^\times_\kappa$), and the restriction of the obstruction theory
$\bE^\bullet$ can be made $\bC^\times$-equivariant as well. The
following is our main conjecture.

\begin{conjecture} \label{conj:DT-PT-Ovir}
  There exists a canonical halving
  \begin{equation} \label{eq:canonical-halving}
    \bE_{1/2} + \bE_{1/2}^\vee = \bE^\bullet\Big|_{M^{\sA \times \bC^\times}} \in K(M^{\sA \times \bC^\times}) \otimes \bk_{\sA \times \bC^\times,\loc},
  \end{equation}
  given explicitly in Definition~\ref{def:canonical-half}, such that
  \begin{equation} \label{eq:DT-PT-Ovir}
    \left(\iota_*\right)^{-1}\hat\cO^\vir_M \stackrel{?}{=} \left[\frac{\hat\cO^{\vir,\mathrm{BF}}_{M^{\sA \times \bC^\times}}}{\hat\se_{\bC^\times}(\cN^{\vir,\mathrm{BF}})} \frac{\hat \se_{\bC^\times}(\cT^{\vir,\mathrm{BF}}_{M^\sA})}{\hat \se_{\bC^\times}(\bE_{1/2}^{\sA\text{-fix}})}\frac{1}{\hat \se_{\sA \times \bC^\times}(\bE_{1/2}^{\sA\text{-mov}})}\right]\Bigg|_{\sA} \in K(M^\sA) \otimes \bk_{\sA,\loc}
  \end{equation}
  where the second fraction is a sign $\pm 1$ on any given connected
  component $F \subset M^{\sA \times \bC^\times}$.
\end{conjecture}

The term inside the brackets lives on
$M^{\sA \times \bC^\times} = M^{\sT}$ and is implicitly pushed forward
from there. We refer to the restriction of the second fraction above
to a connected component $F\subset M^{\sA \times \bC^\times}$ as the
{\it sign} of $F$, e.g. see Example~\ref{ex:CY-fixed-locus-P2}.

In other words, \eqref{eq:DT-PT-Ovir} states that the Oh--Thomas
virtual cycle $\hat\cO^\vir_{M^\sA}$ on the fixed locus $M^\sA$ agrees
with the $\bC^\times$-localized virtual cycle 
\begin{equation} \label{eq:behrend-fantechi-cycle-localization}
  \left[\frac{\hat\cO^{\vir,\mathrm{BF}}_{M^{\sA \times \bC^\times}}}{\hat \se_{\bC^\times}(\cN^{\vir,\mathrm{BF}})}\right]\Bigg|_\sA = \hat\cO^{\vir,\mathrm{BF}}_{M^\sA} \in K(M^\sA) \otimes \bk_{\sA,\loc}
\end{equation}
up to a sign at each $\bC^\times$-fixed component $F$. While the
restriction to $\sA$ in \eqref{eq:behrend-fantechi-cycle-localization}
is well-defined by properness of $M^\sA$, it is part of the conjecture
that the same restriction in \eqref{eq:DT-PT-Ovir} is well-defined.
Note that only $\hat\cO^{\vir,\mathrm{BF}}_{M^\sA}$, not
$\hat\cO^\vir_{M^\sA}$, admits the usual $\bC^\times$-equivariant
virtual localization formula \cite{Graber1999}, but the splitting
\eqref{eq:canonical-halving} which specifies these signs is only
defined on $M^{\sA \times \bC^\times}$, so it is important that
$M^\sA$ has the Behrend--Fantechi virtual cycle
\eqref{eq:PT-fixed-locus-BF-cycle}.

\subsubsection{}

The vertices $V^{\DT}$ and $V^{\PT}$ are defined by separating vertex
and edge contributions in
\[ \bE^\bullet\big|_I = \sum_\alpha \sV(I\big|_{U_\alpha}; \vec t_\alpha) + \sum_{\alpha \to \beta} \sE(I\big|_{U_{\alpha\beta}}; \vec t_{\alpha\beta}) \in \bk_{\sT,\loc} \]
and our canonical halving \eqref{eq:canonical-halving} similarly
arises from canonical halvings $\sV_{1/2}$ and $\sE_{1/2}$ of $\sV$
and $\sE$. Therefore, along with Proposition~\ref{prop:fixed-loci},
our main conjecture gives an explicit way to compute DT and PT
vertices. This was previously infeasible because the complexity of
$\sA$-fixed PT loci prohibited any real understanding of the
Oh--Thomas virtual cycle on it. We checked the following using
SageMath code \cite{Liu}, using the results of
\S\ref{sec:explicit-computation} to simplify computation.

\begin{proposition} \label{prop:DT-PT-check}
  Assuming Conjecture~\ref{conj:DT-PT-Ovir}, the DT/PT vertex
  correspondence \eqref{eq:DT-PT-vertex-correspondence} holds if:
  \begin{itemize}
  \item there is $1$ non-trivial leg of size $< 9$, modulo $Q^8$;
  \item there are $2$ non-trivial legs of total size $< 8$, modulo $Q^7$;
  \item there are $3$ non-trivial legs of total size $< 8$, modulo $Q^6$;
  \item there are $4$ non-trivial legs of total size $< 7$, modulo $Q^5$.
  \end{itemize}
\end{proposition}

It turns out the DT/PT vertex correspondence is quite a strong litmus
test for whether our main conjecture prescribes the correct
halves/signs for both $V^{\DT}$ and $V^{\PT}$; see the uniqueness
claim of \cite[Proposition 1.17]{Cao2022}.
Proposition~\ref{prop:DT-PT-check} is especially non-trivial in light
of the complicated structure of $\sA$-fixed PT loci.

\subsubsection{}

Previous studies \cite{Nekrasov2020, Nekrasov2019, Cao2022,
  Monavari2022} of $4$-fold DT and PT vertices had restrictions on the
number of non-trivial legs $\pi^1, \pi^2, \pi^3, \pi^4$, so that in
particular $M^\sA = M^{\sA \times \bC^\times}$ and is a collection of
isolated points. In such a special case, \eqref{eq:DT-PT-Ovir}
simplifies to
\[ (\iota_*)^{-1}\hat\cO^\vir_M \stackrel{?}{=} \frac{\cO_{M^\sA}}{\hat \se_\sA(\bE_{1/2})}, \]
which, comparing with \eqref{eq:oh-thomas-cycle}, is just the
conjecture
$\sqrt{\se_\sA}(\bE) \stackrel{?}{=} \hat \se_\sA(\bE_{1/2})$.
However, these previous works choose a {\it different} explicit half
$\sV_{1/2}^{\text{CKM}}$ than our $\sV_{1/2}$, and require an
additional sign prescription for each fixed point in $M^\sA$.
Explicitly, for DT moduli spaces, the conjectural sign rule
\cite[Remark 1.18]{Cao2022} for a box configuration $\xi$ with empty
fourth leg is the {\it Nekrasov--Piazzalunga sign}
$(-1)^{\NP_4(\xi)}$. In \S\ref{sec:NP-sign}, we define both
$\sV_{1/2}^{\text{CKM}}$ and $\NP_4$, and show that our main
conjecture subsumes this sign prescription as follows.

\begin{proposition} \label{prop:CKM-vs-our-half}
  Let $\xi$ be a solid partition with legs with empty fourth leg. Then
  \begin{equation} \label{eq:CKM-vs-our-half}
    \frac{(-1)^{\NP_4(\xi)}}{\hat \se_\sA(\bar\sV_{1/2}^{\text{CKM}}(\xi))} = \frac{1}{\hat \se_\sA(\bar\sV_{1/2}(\xi))} \in \bk_{\sA,\loc}
  \end{equation}
  where $\bar \sV$ denotes the restriction $\sV\big|_\sA$.
\end{proposition}

In the case of solid partitions {\it without} legs, i.e.
$\Hilb(\bC^4)$, the lhs of \eqref{eq:CKM-vs-our-half} is known to be
equal to $1/\sqrt{\se_\sA}(\bE)\big|_\xi$ by explicit computation with
a global Oh--Thomas presentation \cite{Kool}, and hence
Proposition~\ref{prop:CKM-vs-our-half} proves our main conjecture. It
would be interesting to understand the geometric interpretation, if
any, of our $\sV_{1/2}$ in Kool and Rennemo's setup.

\subsubsection{}

In \S\ref{sec:1-leg}, we also give some compelling evidence for the
correctness of our half $\bE_{1/2}$ in the case of the PT vertex with
one non-trivial leg. In this setting, the PT moduli space may be
identified with a moduli space of {\it quasimaps}, and under this
identification there is a certain half of the quasimap obstruction
theory with geometric significance. Our explicit formula for
$\bE_{1/2}$ in Definition~\ref{def:canonical-half} was originally
motivated by this $1$-leg calculation.

Many new developments in the $1$-legged case have appeared very
recently \cite{Cao2023,Piazzalunga2023}, and it would be interesting
them to compare our results.

\subsubsection{}

Finally, in \S\ref{sec:explicit-computation}, we address the problem
of explicitly computing the symmetrized Behrend--Fantechi virtual
cycle \eqref{eq:PT-fixed-locus-BF-cycle} on components of PT fixed
loci. This is a necessary step for computation of terms in the vertex
using Conjecture~\ref{conj:DT-PT-Ovir}. Using properties of the quiver
representations $\bar M$ and quiver Grassmannians $\Gr(\bar M)$
appearing in Proposition~\ref{prop:fixed-loci}, we give
characterizations in special cases of whether this virtual cycle is:
\begin{enumerate}
\item (\S\ref{sec:localizable-fixed-loci}) amenable to a further
  virtual equivariant localization, with isolated fixed points;
\item (\S\ref{sec:unobstructed-fixed-loci}) unobstructed, in the sense
  that it equals $\hat\cO \coloneqq \cO \otimes \det(\cT)^{-1/2}$.
\end{enumerate}
Both cases significantly simplify integration over $\Gr(\bar M)$ with
this virtual cycle. In fact, case (i) is already sufficient for the
low-degree computations of Proposition~\ref{prop:DT-PT-check}.

\subsection{Acknowledgements}

This project originated from many inspiring discussions with Martijn
Kool, Sergej Monavari, and Reinier Schmiermann, and was supported by
the Simons Collaboration on Special Holonomy in Geometry, Analysis and
Physics.

\section{Fixed loci}
\label{sec:fixed-loci}

\subsection{Quiver Grassmannians}

\subsubsection{}

\begin{definition}
  Let $Q$ be a connected quiver and let $M$ be a representation of $Q$
  in finite-dimensional $\bC$-vector spaces. Write $v \in Q$ (resp.
  $[v \to w] \in Q$) to mean that $Q$ has a node $v$ (resp. edge from
  $v$ to $w$), and similarly let $M_v$ (resp. $M_{v \to w}$) denote
  the vector spaces (resp. linear maps) in $M$. Let
  $\dim M \coloneqq (\dim M_v)_v$ denote the dimension vector. Both
  $Q$ and $\dim M$ are allowed to be infinite.

  For a finite dimension vector $\vec e$, in complete analogy with the
  classical Grassmannian, let
  \[ \Gr_{\vec e}(M) \coloneqq \{N \subset M : \dim N = \vec e\} \]
  be the {\it quiver Grassmannian} of quiver sub-representations
  $N \subset M$. It is manifestly a closed subscheme of products of
  classical Grassmannians $\prod_v \Gr_{e_v}(M_v)$, and is therefore a
  projective scheme. We refer to $M$ as the {\it ambient quiver
    representation}. For short, write
  \[ \Gr(M) \coloneqq \bigsqcup_{\vec e} \Gr_{\vec e}(M). \]
  Let $\scN$ and $\scM$ denote the universal quiver representations on
  $\Gr_{\vec e}(M)$, corresponding to $N$ and $M$ respectively.
\end{definition}

\subsubsection{}

Quiver Grassmannians also have a GIT presentation
\begin{equation} \label{eq:quiver-grassmannian-GIT-presentation}
  \Gr_{\vec e}(M) = Z(R_{\vec e, M}) \mathbin{\Big/\mkern-9mu\Big/}_{\!\!\!\theta}\; \GL(\vec e)
\end{equation}
where $\GL(\vec e) \coloneqq \prod_v \GL(e_v)$, the prequotient is the
zero locus of
\begin{align*}
  R_{\vec e, M}\colon \begin{array}{c} \bigoplus_{v \to w} \Hom(\bC^{e_v}, \bC^{e_w}) \\ \oplus \bigoplus_v \Hom(\bC^{e_v}, M_v)\end{array} &\to \bigoplus_{v \to w} \Hom(\bC^{e_v}, M_w) \\
  (N, \phi) &\mapsto (\phi_w \circ N_{v \to w} - M_{v \to w} \circ \phi_v)_{v \to w},
\end{align*}
and the GIT stability condition $\theta$ is chosen so that all
$\phi_v$ are injective \cite[Lemma 2]{Caldero2008}. Clearly
$R_{\vec e, M}$ is equivariant for the action $\alpha$ of
$\GL(\vec e)$.

\begin{proposition}
  $\Gr_{\vec e}(M)$ has a perfect obstruction theory given by
  $(\bE^\bullet[1])^\vee$ where
  \begin{equation} \label{eq:quiver-grassmannian-obstruction-theory}
    \bE^\bullet \coloneqq \bigg[\bigoplus_v \Hom(\scN_v, \scN_v) \xrightarrow{d\alpha} \begin{array}{c} \bigoplus_{v \to w} \Hom(\scN_v, \scN_w) \\ \oplus \bigoplus_v \Hom(\scN_v, \scM_v)\end{array} \xrightarrow{dR_{\vec e,M}} \bigoplus_{v \to w} \Hom(\scN_v, \scM_w)\bigg].
  \end{equation}
\end{proposition}

\begin{proof}
  The GIT quotient is a zero locus inside a smooth scheme, so the
  conclusion follows from the basic case of \cite[\S 2]{Graber1999}.
  In more modern language, there are inclusions
  \[ \Gr_{\vec e}(M) \hookrightarrow [Z(R_{\vec e, M})/\GL(\vec e)] \hookrightarrow [Z^{\der}(R_{\vec e, M})/\GL(\vec e)], \]
  first as an open locus in the ambient stack, and then as the
  classical truncation of the corresponding {\it derived} stack where
  $Z^{\der}$ denotes derived zero locus. The cotangent complex of the
  derived stack is manifestly $(\bE^\bullet[1])^\vee$ and is perfect
  in amplitude $[-1, 1]$. By standard considerations, see e.g.
  \cite[\S 1]{Schuerg2015}, its restriction to $\Gr_{\vec e}(M)$ is a
  perfect obstruction theory.
\end{proof}

\subsubsection{}

A morphism $\phi \in \Hom_Q(M, N)$ of quiver representations is a
collection of linear maps $(\phi_v)_v \in \prod_v \Hom(M_v, N_v)$
intertwining the linear maps in $M$ and $N$. Using the GIT
presentation \eqref{eq:quiver-grassmannian-GIT-presentation}, by the
same argument as for the classical Grassmannian \cite[Proposition
6]{Caldero2008},
\[ T_N \Gr_{\vec e}(M) = \Hom_Q(N, M/N) \]
and smoothness at $N \in \Gr_{\vec e}(M)$ is measured by the vanishing
of $\Ext^1_Q(N, M/N)$.

\subsubsection{}

\begin{definition} \label{def:ambient-quiver-notation}
  We introduce some notation for the ambient quiver representations
  which will appear later from PT fixed loci. For $i \in \{1,2,3,4\}$,
  let $\bC_i \coloneqq \bC e_i$ for short, and for a subset
  $S \subset \{1,2,3,4\}$ with $|S| > 1$, let
  \[ \bar \bC_S \coloneqq \bigoplus_{i \in S} \bC e_i \Big/ \bC \sum_{i \in S} e_i. \]
  Such vector spaces will form the nodes of the ambient quiver
  representation $\bar M$, and will be connected by the obvious
  projection and inclusion maps
  \begin{align*}
    \pi_{S,S'}\colon \bar \bC_S &\twoheadrightarrow \bar \bC_{S'}, \qquad S' \subset S \\
    \iota_j\colon \bC_j &\hookrightarrow \bar \bC_S, \qquad j \in S.
  \end{align*}
\end{definition}

Our quiver $Q$ and $\bar M$ will generally be infinite, but we will
only consider sub-representations $\bar N \subset \bar M$ of finite
dimension vector. We only draw the relevant part of $\bar M$ where the
$\bar N$ under consideration is non-empty.

\subsubsection{}

It turns out that {\it any} projective variety can be written as a
quiver Grassmannian \cite{Reineke2013} (using an acyclic quiver with
at most three vertices!) so no special properties should be expected
from $\Gr_{\vec e}(M)$ unless $M$ or $\vec e$ is very special. In
particular, the quiver Grassmannians appearing later can be be
arbitrarily singular, non-reduced and not of pure dimension.

\begin{example} \label{ex:quiver-grassmannian-singular}
  Following the notation of
  Definition~\ref{def:ambient-quiver-notation}, the quiver
  sub-representations $\bar N \subset \bar M$ given by
  \[ \begin{tikzcd}[column sep=1em, row sep=1em]
      0 \ar[hookrightarrow]{rr} && \bar \bC_{13} \\
      & \bC \ar[hookrightarrow]{rr} && \bar \bC_{123} \\
      \bC \ar{uu} \ar{ur} \ar[hookrightarrow]{rr}{i} && \bar \bC_{1234} \ar{ur} \ar{uu}
    \end{tikzcd} \]
  form a component $\Gr_{(1,1,0)}(\bar M)$ isomorphic to the union of two
  $\bP^1$ at a point where $i(\bC) = \bC_4$. Labeling the nodes of $\bar N$
  as $\bar N_1 \to \bar N_2$ and $\bar N_3 = 0$,
  \eqref{eq:quiver-grassmannian-obstruction-theory} gives
  \[ \cT^\vir \Gr_{(1,1,0)}(M) = \scN_1^{-1} \scN_2 - 2 + 2\scN_2^{-1} \in K^\circ(\Gr_{(1,1,0)}(M)). \]
\end{example}

\subsection{For the full torus}
\label{sec:PT-full-fixed-loci}

\subsubsection{}

Fix a toric chart $\bC^4 \subset X$ and assume its coordinates
$x_1, x_2, x_3, x_4$ have $\sT$-weight $t_1, t_2, t_3, t_4$. In this
subsection, we characterize the restriction
$s\colon \cO_{\bC^4} \to \cF_{\bC^4}$ of a $\sT$-fixed stable pair to
$\bC^4$. The first step is completely analogous to the $3$-fold case
\cite[\S 2]{Pandharipande2009a}. The CM support curve $C$ of this PT
stable pair is cut out by a monomial ideal
\[ \cI_C \subset \bC[x_1, x_2, x_3, x_4] \]
which is a $\sT$-fixed point in DT moduli space and so corresponds to
a solid partition with legs. Namely, treating indices cyclically
modulo $4$, the $i$-th leg $\pi^i$ is the plane partition in
coordinates $x_{i+1}, x_{i+2}, x_{i+3}$ such that
\[ (\cI_C)_{x_i} = \frac{\bC[x_1, x_2, x_3, x_4]_{x_i}}{\pi^i[x_{i+1}, x_{i+2}, x_{i+3}] \cdot \bC[x_1, x_2, x_3, x_4]_{x_i}}. \]
Note that not all the $\pi^i$ can be empty since $C$ has dimension
$1$. Set
\[ M \coloneqq \bigoplus_{i=1}^4 M_i, \qquad M_i \coloneqq (\cO_{\bC^4})_{x_i}/(\cI_C)_{x_i}. \]
In the language of box configurations, where a $\sT$-invariant
$\cO_{\bC^4}$-module is described in terms of its $\sT$-weights in the
weight space $\bZ^4$ of $\sT$, the module $M_i$ is an infinite
cylinder along the $x_i$ axis with cross-section $\pi^i$, extending in
both the positive and negative $x_i$ directions.

\begin{proposition}[{\cite[Proposition 1]{Pandharipande2009a}}]
  $\sT$-fixed stable pairs on $\bC^4$ correspond to finitely-generated
  $\sT$-invariant $\cO_{\bC^4}$-submodules
  \[ Z \subset \bar M \coloneqq M/\inner{(1,1,1,1)}. \]
\end{proposition}

\subsubsection{}
\label{sec:PT-full-quiver}

The $\sT$-fixed PT locus on $\bC^4$ is therefore the space of these
submodules $Z \subset \bar M$. Attempting to characterize $Z$ in the
language of the {\it labeled} box configurations of the $3$-fold case
will not go well due to new phenomena specific to $4$-folds; see
Examples~\ref{ex:PT-full-fixed-locus-non-labelable} and
\ref{ex:PT-full-fixed-locus-Quot-scheme} below. Instead, observe that
the $\cO_{\bC^4}$-module $\bar M$ is equivalently a quiver
representation: nodes are $\sT$-weight spaces in $\bar M$, and
multiplication by coordinates $x_i$ induces linear maps between nodes.
Put differently, $\bar M$ has a $\bZ^4$-grading from $\sT$, and the
underlying quiver $Q$ has a vertex $v \in \bZ^4$ iff the graded piece
$\bar M_v$ is non-empty, and an edge $v \to w$ iff $t(w) = t(v)t_i$
for some $i$, where $t(v)$ is the $\sT$-weight corresponding to $v$.

As there is little risk of confusion, we use $\bar M$ to denote both
the module and the quiver representation. The quiver $Q$ and its
representation $\bar M$ are connected and acyclic but always infinite.
We use the notation of Definition~\ref{def:ambient-quiver-notation}
for nodes and edges of $\bar M$, so that $\bar M_v = \bar \bC_S$ is
spanned by the generators $e_i \in (M_i)_v$ for $i \in S$.

Clearly, finitely-generated $\sT$-invariant $\cO_{\bC^4}$-submodules
$Z \subset \bar M$ correspond to quiver sub-representations
\[ \bar N \in \Gr(\bar M) \]
with finite dimension vector. Hence $\sT$-fixed PT loci on $\bC^4$ are
quiver Grassmannians once one checks
(Proposition~\ref{prop:PT-fixed-locus-tangent-space}) that they have
the same Zariski tangent space, to rule out possible scheme-theoretic
thickening.

\subsubsection{}

\begin{remark}
  This quiver Grassmannian description of $\sT$-fixed PT loci works
  equally well in any dimension. For $3$-folds, we can match it with
  the previous combinatorial description \cite{Pandharipande2009a} in
  terms of type I, II, and III boxes. The quiver representation
  $\bar M$ has nodes which are either $\bC$ (locations of type I${}^-$
  or type II boxes) or $\bC^2$ (locations of type III boxes). In the
  quiver Grassmannian $\Gr(\bar M)$, positive-dimensional components
  can only occur by asking for $1$-dimensional subspaces
  $\bC \subset \bC^2$ (labeled type III boxes) in the quiver
  sub-representation $\bar N \subset \bar M$. Since all maps between
  the $\bC^2$ are isomorphisms, so are all maps between such $\bC$.
  Hence $\Gr_{\vec e}(\bar M)$ for any given dimension vector $\vec e$
  is isomorphic to a product of $\Gr(1, \bC^2) \cong \bP^1$, the
  connected components (unrestricted path components) of
  freely-varying $\bC \subset \bC^2$ in $\bar N$.
\end{remark}

\subsubsection{}

\begin{example} \label{ex:PT-full-fixed-locus-non-labelable}
  Take the legs $\pi^1, \pi^2, \pi^3, \pi^4$ to be
  \[ \includegraphics[scale=1.5]{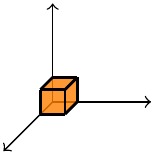} \quad \includegraphics[scale=1.5]{images/p1.pdf} \quad \includegraphics[scale=1.5]{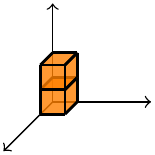} \quad \includegraphics[scale=1.5]{images/p2y.pdf} \]
  respectively, with axes ordered cyclically counterclockwise starting
  from the horizontal one. Then $\bar M$ is
  \[ \begin{tikzcd}[column sep={1.5em,between origins}, row sep={3.15em,between origins}]
    {} &&&&&& \bar \bC_{24} &&&&&& \bar \bC_{13} \\
    {} &&& \bC_4 \ar{urrr} &&&&&& \bar \bC_{1234} \ar{ulll} \ar{urrr} &&&&&& \bC_3 \ar{ulll} \\
    \rotatebox[origin=c]{35}{$\cdots$} \ar{urrr} &&&&&& \bC_4 \ar{ulll} \ar{urrr} && \bC_1 \ar{ur} && \bC_2 \ar{ul} && \bC_3 \ar{ulll} \ar{urrr} &&&&&& \rotatebox[origin=c]{-35}{$\cdots$} \ar{ulll} \\
    {} &&& \rotatebox[origin=c]{35}{$\cdots$} \ar{urrr} &&&& \rotatebox[origin=c]{64.5}{$\cdots$} \ar{ur} &&&& \rotatebox[origin=c]{-64.5}{$\cdots$} \ar{ul} &&&& \rotatebox[origin=c]{-35}{$\cdots$} \ar{ulll}
  \end{tikzcd} \]
  Let $\vec e$ be the dimension vector which is $1$ at the central
  $\bar \bC_{1234}$ node and $0$ elsewhere. Then
  \[ \Gr_{\vec e}(\bar M) = \{\bar N\} \]
  where $\bar N \subset \bar M$ is the sub-representation which is
  zero everywhere except at
  \[ \bC = \ker \pi_{1234,13} \cap \ker \pi_{1234,24} \subset \bar \bC_{1234}. \]
  This $\bC$ is generated by the vector $(1,-1,1,-1)$, which is
  notably not a standard basis vector nor equivalent to one modulo
  $(1,1,1,1)$. This is a $4$-fold phenomenon which does not occur for
  $3$-folds.
\end{example}

\subsubsection{}

\begin{example} \label{ex:PT-full-fixed-locus-Quot-scheme}
  Example~\ref{ex:quiver-grassmannian-singular} is a PT fixed locus
  for appropriate legs, and already demonstrates that, unlike for
  $3$-folds, $4$-fold PT fixed loci can be singular. In fact, in stark
  contrast with $3$-folds and in full accordance with Murphy's Law
  \cite{Vakil2006}, they can be {\it arbitrarily} singular. We show
  this by ``embedding'' $\sT$-fixed loci of
  $\Quot(\cO_{\bC^4}^{\oplus 3})$, which are known to have arbitrary
  singularities \cite[Theorem 2.1.1]{Schmiermann2021}, into our
  $\sT$-fixed PT loci as follows. View $\cO_{\bC^4}^{\oplus 3}$ as the
  representation $M^{\Quot}$ of the $4$-dimensional (positive) lattice
  quiver where each node is a $\bC^3$ and all edges are isomorphisms.
  Then an element
  \[ [s\colon \cO_{\bC^4}^{\oplus 3} \twoheadrightarrow \cF] \in \Quot(\cO_{\bC^4}^{\oplus 3})^\sT \]
  is equivalently a quiver sub-representation
  $N \coloneqq \ker(s) \subset M^{\Quot}$ of finite codimension. But,
  taking the four legs $\pi^1 = \pi^2 = \pi^3 = \pi^4$ to be
  sufficiently big, our $\bar M$ also contains the interesting portion
  of $M^{\Quot}$ where $N$ differs from $M^{\Quot}$. Hence $N$ is also
  a quiver sub-representation of our $\bar M$, and we are done.
\end{example}

\subsection{For the Calabi--Yau torus}
\label{sec:PT-CY-fixed-loci}

\subsubsection{}

Now suppose $s\colon \cO_X \to \cF$ is a $\sA$-fixed instead of
$\sT$-fixed stable pair. The support curve $C$ is a priori
only $\sA$-invariant, but since $[\cI_C]$ is a point in the DT moduli
space and
\begin{equation} \label{eq:DT-CY-vs-full-torus}
  \DT_{\beta, n}(X)^\sA = \DT_{\beta, n}(X)^\sT,
\end{equation}
it follows that $C$ is actually $\sT$-invariant \cite[Lemma
2.1]{Cao2020}. So, on $\bC^4 \subset X$, we continue to think about
$4$-tuples of plane partitions $(\pi^i)_{i=1}^4$ forming the legs of a
box configuration.

\subsubsection{}
\label{sec:PT-CY-quiver}

As before, let $M_i \coloneqq (\cO_{C_{\pi^i}})_{x_i}$ and
$M \coloneqq \bigoplus_{i=1}^4 M_i$, so that $\sA$-fixed stable pairs
correspond to finitely-generated $\sA$-invariant
$\cO_{\bC^4}$-submodules
\[ Z \subset \bar M^{\CY} \coloneqq M/\inner{(1,1,1,1)}. \]
Here, the superscript $\CY$ indicates that this module is only
$\bZ^3$-graded instead of $\bZ^4$-graded like $\bar M$. The
$\sA$-weight spaces of $\bar M^{\CY}$ form a quiver representation of
a quiver $Q^{\CY}$, and the desired submodules $Z$ are then quiver
sub-representations of $\bar M^{\CY}$ with finite dimension vector.

If $Q$ is the quiver underlying the fully $\bZ^4$-graded $\bar M$,
evidently $Q^{\CY}$ is the quotient of $Q$ by the equivalence relation
\[ v \sim w \iff t(v) \equiv t(w) \bmod{\kappa} \]
on nodes of $Q$, and similarly the nodes of $\bar M^{\CY}$ are given
by
\[ \bar M^\CY_v = \bigoplus_{\substack{w \in Q\\v \sim w}} \bar M_w. \]
Edges/maps are the induced ones. Note that $Q^{\CY}$ may not be
acyclic anymore, although $\bar M^{\CY}$ is still acyclic in the sense
that the composition of all maps in any cycle is always zero.

\subsubsection{}
\label{sec:Cstar-action-on-CY-fixed-loci}

One can put a $\bC^\times$-action on $\Gr(\bar M^{\CY})$ by
reintroducing the grading by $\kappa$, namely $s \in \bC^\times$
acts by
\[ s \cdot \bar M^{\CY}_v \coloneqq \bigoplus_{v \sim w} s^{\deg_\kappa t(w)} \bar M_w \]
with induced action on the maps $\bar M_{v \to w}^{\CY}$. Clearly the
discrepancy between $\Gr(\bar M^{\CY})$ and
$\Gr(\bar M^{\CY})^{\bC^\times} = \Gr(\bar M)$ is the discrepancy
between $\sA$-fixed and $\sT$-fixed PT loci. Since torus-fixed loci of
smooth schemes are themselves smooth, singularities in $\Gr(\bar M)$
(Example~\ref{ex:PT-full-fixed-locus-Quot-scheme}) induce
singularities in $\Gr(\bar M^{\CY})$, so $\sA$-fixed PT loci are in
general singular too.

For up to two non-trivial legs, $\sA$-fixed and $\sT$-fixed loci in
the PT moduli space agree and are just isolated points
\cite[Propositions 2.5, 2.6]{Cao2020}, but in general and in contrast
to \eqref{eq:DT-CY-vs-full-torus},
\[ \PT_{\beta,n}(X)^\sA \supsetneq \PT_{\beta,n}(X)^\sT. \]
The following Example~\ref{ex:PT-CY-vs-full-fixed-locus} of this
phenomenon has four non-trivial legs and is in some sense the
next-to-minimal example.

\subsubsection{}

\begin{example} \label{ex:PT-CY-vs-full-fixed-locus}
  Take the legs $\pi^1, \pi^2, \pi^3, \pi^4$ to be
  \[ \includegraphics[scale=1.5]{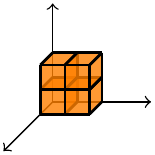} \quad \includegraphics[scale=1.5]{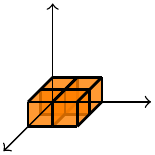} \quad \includegraphics[scale=1.5]{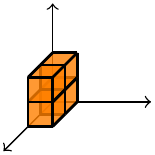} \quad \includegraphics[scale=1.5]{images/p1.pdf} \]
  respectively, with axes ordered cyclically counterclockwise. Then
  \[ \bar M =
    \begin{tikzcd}[column sep=1em, row sep=1em, baseline={([yshift=3.5em]origin.base)}]
      & \bar\bC_{123} \ar{rr} && \bar\bC_{123} \\
      \bar\bC_{123} \ar{rr} \ar{ur} && \bar\bC_{123} \ar{ur} \\
      & \bar\bC_{123} \ar{uu} \ar{rr} && \bar\bC_{123} \ar{uu} \\
      |[alias=origin]| \bar\bC_{1234} \ar{ur} \ar{uu} \ar{rr} && \bar\bC_{123} \ar{uu} \ar{ur} \\
      & \bC_4 \ar{ul}
    \end{tikzcd}, \qquad
    \bar M^{\CY} =
    \begin{tikzcd}[column sep=1em, row sep=1em, baseline={([yshift=3.5em]origin.base)}]
      & \bar\bC_{123} \ar{rr}{\begin{psmallmatrix} 1\\ 0\end{psmallmatrix}} && \bar\bC_{123} \oplus \bC_4 \ar[bend left=20]{dddlll}{(0, \iota_4)} \\
      \bar\bC_{123} \ar{rr} \ar{ur} && \bar\bC_{123} \ar{ur}{\begin{psmallmatrix} 1\\ 0\end{psmallmatrix}} \\
      & \bar\bC_{123} \ar{uu} \ar{rr} && \bar\bC_{123} \ar[swap]{uu}{\begin{psmallmatrix} 1\\ 0\end{psmallmatrix}} \\
      |[alias=origin]| \bar\bC_{1234} \ar{ur} \ar{uu} \ar{rr} && \bar\bC_{123} \ar{uu} \ar{ur}
    \end{tikzcd} \]
  There is a $\bP^2$ component in $\Gr(\bar M^{\CY})$ consisting of
  quiver sub-representations
  \[ \bar N^{\CY} =
    \begin{tikzcd}[column sep=1em, row sep=1em, baseline={([yshift=3.5em]origin.base)}]
      & 0 \ar{rr} && \bC \ar[bend left=20]{dddlll} \\
      0 \ar{rr} \ar{ur} && 0 \ar{ur} \\
      & 0 \ar{uu} \ar{rr} && 0 \ar{uu} \\
      |[alias=origin]| \iota_4(\bC_4) \ar{ur} \ar{uu} \ar{rr} && 0 \ar{uu} \ar{ur}
    \end{tikzcd}
    \subset \bar M^{\CY} \]
  where the top right $\bC \subset \bar\bC_{123} \oplus \bC_4$ is an
  arbitrary subspace. The two $\sT$-fixed loci are
  \[ \bar N^{(1)} =
    \begin{tikzcd}[column sep=1em, row sep=1em, baseline={([yshift=3.5em]origin.base)}]
      & 0 \ar{rr} && 0 \\
      0 \ar{rr} \ar{ur} && 0 \ar{ur} \\
      & 0 \ar{uu} \ar{rr} && 0 \ar{uu} \\
      |[alias=origin]| \iota_4(\bC_4) \ar{ur} \ar{uu} \ar{rr} && 0 \ar{uu} \ar{ur} \\
      & \bC_4 \ar{ul}
    \end{tikzcd}, \qquad
    \bar N^{(2)} =
    \begin{tikzcd}[column sep=1em, row sep=1em, baseline={([yshift=3.5em]origin.base)}]
      & 0 \ar{rr} && \bC \\
      0 \ar{rr} \ar{ur} && 0 \ar{ur} \\
      & 0 \ar{uu} \ar{rr} && 0 \ar[swap]{uu} \\
      |[alias=origin]| \iota_4(\bC_4) \ar{ur} \ar{uu} \ar{rr} && 0 \ar{uu} \ar{ur}
    \end{tikzcd} \subset \bar M \]
  corresponding to a point and a $\bP^1$ where the top right
  $\bC \subset \bar \bC_{123}$ is an arbitrary subspace.
\end{example}

\subsubsection{}

\begin{proposition}[{\cite[Lemma 3]{Pandharipande2009a}}] \label{prop:PT-fixed-locus-tangent-space}
  The Zariski tangent space
  $\Ext^0_{\bC^4}(I_{\bC^4}, \cF_{\bC^4})^\sA$ to the $\sA$-fixed
  locus on $\bC^4$ is equal to
  $\Hom_{\bC^4}(\cQ_{\bC^4}, M/\cF_{\bC^4})^\sA$.
\end{proposition}

Clearly $M/\cF_{\bC^4} \cong \bar M/\cQ_{\bC^4}$, and as an
$\sA$-module $\bar M$ corresponds to the quiver $\bar M^\CY$.
Similarly $\cQ_{\bC^4}$ corresponds to a quiver sub-representation
$\bar N^\CY \subset \bar M^\CY$. It follows that
\[ \Hom_{\bC^4}(\cQ_{\bC^4}, M/\cF_{\bC^4})^\sA = \Hom_{Q^\CY}(\bar N^\CY, \bar M^\CY/\bar N^\CY) = T_{\bar N^\CY}\Gr(\bar M^\CY). \]
Hence $\sA$-fixed PT loci on $\bC^4$ really are quiver Grassmannians,
not some thickening thereof. This is therefore true for the
$\sT$-fixed PT loci of \S\ref{sec:PT-full-fixed-loci} as well. The
local-to-global isomorphism
\[ \Ext^0(I, \cF)^\sA \cong \bigoplus_\alpha \Ext^0\big(I\big|_{U_\alpha}, \cF\big|_{U_\alpha}\big)^\sA \]
for $\sA$-fixed PT pairs \cite[Equation (3.5)]{Pandharipande2009a}
therefore implies that $\sA$-fixed PT loci on $X$ are products of
quiver Grassmannians as well.

\section{Obstruction theory and halves}

\subsection{Canonical half}
\label{sec:canonical-half}

\subsubsection{}

Following \S\ref{sec:PT-full-fixed-loci}, fix a connected component $Z
= Z_{\bC^4} \times \cdots \subset \PT_{\beta,n}(X)^\sA$ and let
$[\cO_{\bC^4} \xrightarrow{s} \cF_{\bC^4}]$ be the restriction of a PT
pair in $Z$ to the chart $\bC^4 \subset X$. Let $\bC^\times$ be
another copy of the remaining $\bC_\kappa^\times$-action on
$Z_{\bC^4}$, as in \S\ref{sec:Cstar-action-on-CY-fixed-loci}, with
weight denoted by $\rho$ instead of $\kappa$. We require this
$\bC^\times$ because later $\bC_\kappa^\times$ will disappear in the
Calabi--Yau specialization of \eqref{eq:Vhalf-CY-specialization}, but
we still want to perform a sort of $\bC^\times$-equivariant
localization on $Z_{\bC^4}$.

The equivariant character of $\cF_{\bC^4}$ can be written in terms of
universal bundles on $Z_{\bC^4}$, i.e.
\[ \sF \coloneqq \scF|_{Z \times \bC^4} \in K^\circ_{\sT \times \bC^\times}(Z_{\bC^4})_{\loc} \]
where we implicitly identified
$K_{\sT}(\bC^4)_\loc \cong \bk_{\sT,\loc}$. Localization is required
here because $\cF_{\bC^4}$ contains a contribution from $\cO_C$ where
$C \subset \bC^4$ is the CM support curve. While $\sF$ already carries
a canonical $\bC^\times_\kappa$-equivariant structure, we are free to
pick arbitrary non-trivial $\bC^\times$-equivariant structures on
the universal bundles of $Z_{\bC^4}$.

\begin{definition} \label{def:normalized-character}
  The {\it normalized character} is
  \[ \sF_\fin \coloneqq \sF - \sum_i \frac{\sF_i}{1 - t_i} \in K^\circ_{\sT \times \bC^\times}(Z_{\bC^4}) \]
  where $\sF_i/(1-t_i) = \cO_{C_{\pi^i}}$ are the characters of the
  four legs. The sum is the character of the structure sheaf of the
  normalization of $C$, which differs from $\cO_C$ by a finite-length
  module, so localization is no longer required here.
\end{definition}

\subsubsection{}

\begin{definition} \label{def:canonical-half}
  The contribution to $\Ext_X(I, I)_0$ from the chart $\bC^4$ is the
  {\it vertex term}
  \begin{equation} \label{eq:vertex-term}
    \sV \coloneqq \left(\sF + \sF^\vee \kappa_{1234}^\vee - \sF^\vee \sF \tau_{1234}^\vee\right) - \sum_i \frac{\sF_i - \sF_i^\vee \kappa_{jkl}^\vee - \sF_i^\vee \sF_i \tau_{jkl}^\vee}{1 - t_i}
  \end{equation}
  of \cite[\S 4]{Maulik2006} \cite[\S 2.4]{Cao2020}, where the indices
  in the sum satisfy $\{i, j, k, l\} = \{1, 2, 3, 4\}$ and for subsets
  $S \subset \{1,2,3,4\}$ we set
  $\tau_S \coloneqq \prod_{i \in S} (1 - t_i)$ and
  $\kappa_S \coloneqq \prod_i t_i$. Rearranging,
  \begin{align*}
    \sV &= \left(\sF_\fin + \sF_\fin^\vee \kappa_{1234}^\vee\right) - \left(\sF_\fin^\vee \sF_\fin \tau_{1234}^\vee\right) \\
        &\qquad-\sum_i \left(\sF_i^\vee \sF_\fin - \sF_\fin^\vee \sF_i t_i^{-1}\right) \tau_{jkl}^\vee \\
        &\qquad + \sum_{i<j} \left(\sF_i^\vee \sF_j t_j^{-1} + \sF_j^\vee \sF_i t_i^{-1}\right) \tau_{kl}^\vee,
  \end{align*}
  which is manifestly an element of
  $K^\circ_{\sT \times \bC^\times}(Z_{\bC^4})$ and clearly each of the
  bracketed terms is symmetric. We define a half of $\sV$ {\it only on
    the locus} $Z_{\bC^4}^\sT = Z_{\bC^4}^{\bC^\times}$, where there
  is a splitting
  \[ K^\circ_{\sT \times \bC^\times}(Z_{\bC^4}^\sT) = K^\circ(Z_{\bC^4}^\sT) \otimes \bk_{\sT \times \bC^\times} \]
  and so the $\sT$-weight of an element is well-defined. Pick
  $a \in \{1,2,3,4\}$ and a generic cocharacter
  $\sigma\colon \bC^\times \to \sT$ in the cone
  \begin{equation} \label{eq:valid-cocharacters}
    \cochar(\sT)_+ \coloneqq \left\{t_i \gg t_j \gg t_k \gg t_l > 0 : \{i,j,k,l\} = \{1,2,3,4\}\right\}.
  \end{equation}
  Then, on $Z_{\bC^4}^\sT$, define the {\it canonical half}
  \begin{align*}
    \sV_{1/2}^{\sigma,a}
    &\coloneqq \sF_\fin - \left((\sF_\fin^\vee \sF_\fin)_{\ge_\sigma} - ((\sF_\fin^\vee \sF_\fin)_{<_\sigma})^\vee t_a^{-1} \right) \tau_{bcd}^\vee \\
    &\qquad - \sum_i \left(\sF_i^\vee \sF_{\fin,\ge 0} - (\sF_{\fin,<0})^\vee \sF_i t_i^{-1}\right) \tau_{jkl}^\vee \\
    &\qquad + \sum_{i<j} \sF_i^\vee \sF_j t_j^{-1} \tau_{kl}^\vee
  \end{align*}
  where $\{a,b,c,d\} = \{1,2,3,4\}$, subscripts $\ge_\sigma$ (resp.
  $<_\sigma$) mean the part with non-negative (resp. negative)
  $\sigma$-weight, and subscripts $\ge 0$ (resp. $< 0$) mean the part
  with non-negative (resp. non-non-negative) $\sT$-weight. To be clear,
  \[ (t_1^i t_2^j t_3^k t_4^l)_{< 0} = \begin{cases} 0 & i, j, k, l \ge 0 \\ t_1^i t_2^j t_3^k t_4^l & \text{otherwise}. \end{cases} \]
  It is clear that
  $\sV^{\sigma,a}_{1/2} + \kappa^\vee (\sV^{\sigma,a}_{1/2})^\vee = \sV|_{Z_{\bC^4}^\sT}$.
  If unspecified, $a = 4$ and $\sigma$ is reverse lexicographic order,
  i.e. the chamber in \eqref{eq:valid-cocharacters} with
  $(i,j,k,l) = (4,3,2,1)$, following the original convention of
  \cite{Nekrasov2020}.
\end{definition}

\subsubsection{}

Though they were written in the PT setting,
Definitions~\ref{def:normalized-character} and
\ref{def:canonical-half} work equally well for DT moduli spaces. For
both DT and PT, the Calabi--Yau specialization
\begin{equation} \label{eq:Vhalf-CY-specialization}
  \bar \sV_{1/2} \coloneqq \sV_{1/2}\Big|_{\sA \times \bC^\times}
\end{equation}
is the vertex contribution to $\bE_{1/2}$ in our main
Conjecture~\ref{conj:DT-PT-Ovir}. Similarly
\[ \bar \sF_\fin \coloneqq \sF_\fin\Big|_{\sA \times \bC^\times} \]
is the vertex contribution to the Nekrasov insertion $\cO_X^{[n]}$.
Lemmas~\ref{lem:Vhalf-independent-of-a} and
\ref{lem:Vhalf-independent-of-sigma} below show that although
$\bar \sV^{\sigma,a}_{1/2}$ depends on $\sigma$ and $a$, the resulting
formula \eqref{eq:DT-PT-Ovir} of our main conjecture is independent of
$\sigma$ and $a$.

The terminology ``{\it canonical} half'' is because all the terms in
$\sV^{\sigma,a}_{1/2}$ which involve the legs are by definition
independent of $\sigma$ and $a$. Because of rotational symmetry, it
appears difficult to choose a half of the finite term
$\sF^\vee_{\fin} \sF_\fin \tau_{1234}^\vee$ without additional choices
like $\sigma$ or $a$.

\subsubsection{}

\begin{example} \label{ex:CY-fixed-locus-P2}
  Let $\bP^2$ be the $\sA$-fixed component from
  Example~\ref{ex:PT-CY-vs-full-fixed-locus}. Then
  \[ \sF_\fin = \left(1 + \cO_{\bP^2}(-1) \otimes t_4^{-1}\right) - \left(2(1-t_1)(1-t_2)(1-t_3) + 1\right). \]
  Let $p \in \bP^2$ and $L \cong \bP^1 \subset \bP^2$ be the
  $\sT$-fixed components. Then the
  $\bC_\kappa^\times \times \bC^\times$-equivariant structure of
  $\cO_{\bP^2}(-1)$ is determined by
  \[ \cO_{\bP^2}(-1)|_p = \rho^d, \qquad \cO_{\bP^2}(-1)|_L = \cO_L(-1) \otimes \rho^{d+1} \kappa \]
  for any integer $d \in \bZ$. Explicit calculation gives
  \[ \Big(\bar \sV_{1/2}\Big|_p\Big)^{\sA\text{-fix}} = 2\cO_{\bP^2}(1)\Big|_p - 1 + \cO_{\bP^2}(-1)\Big|_p, \qquad \Big(\bar \sV_{1/2}\Big|_L\Big)^{\sA\text{-fix}} = 3\cO_{\bP^2}(1)\Big|_L - 1, \]
  while $\cT^{\vir,\mathrm{BF}}_{\bP^2} = \cT_{\bP^2} = 3\cO(1) - 1$.
  So, in Conjecture~\ref{conj:DT-PT-Ovir}, $p$ has sign $-1$ while $L$
  has sign $+1$.
\end{example}

\subsubsection{}

\begin{lemma} \label{lem:Vhalf-independent-of-a}
  $\hat \se_{\sA \times \bC^\times}(\bar\sV_{1/2}^{\sigma,a})$ is
  independent of $a \in \{1,2,3,4\}$.
\end{lemma}

\begin{proof}
  Without loss of generality, it suffices to compute
  $\hat \se_{\sA \times \bC^\times}$ of the CY specialization of
  \begin{align*}
    \sV_{1/2}^{\sigma,4} - \sV_{1/2}^{\sigma,3}
    &= (\sF_\fin^\vee \sF_\fin)_{\ge_\sigma} (\tau_{123}^\vee - \tau_{124}^\vee) - (\sF_\fin^\vee \sF_\fin)_{>_\sigma} (\tau_{123}^\vee t_4^{-1} - \tau_{124}^\vee t_3^{-1}) \\
    &= (\sF_\fin^\vee \sF_\fin)_{=_\sigma} \tau_{12}^\vee (t_4^{-1} - t_3^{-1}).
  \end{align*}
  This has the form $\sS_{1/2} - \kappa_{1234}^\vee \sS_{1/2}^\vee$.
  Since the cocharacter $\sigma$ is generic,
  $(\sF_\fin^\vee \sF_\fin)_{=_\sigma}$ is constant in $\sT$, so
  $\rank \bar\sS_{1/2}^{\sA\text{-mov}} = \rank \bar \sS_{1/2}$ is
  even and we are done.
\end{proof}

\subsubsection{}

\begin{lemma} \label{lem:Vhalf-independent-of-sigma}
  $\hat \se_{\sA \times \bC^\times}(\bar\sV_{1/2}^{\sigma,a})$ is
  independent of $\sigma \in \cochar(\sT)_+$.
\end{lemma}

\begin{proof}
  Without loss of generality, let
  $\sigma \in \{t_1 \gg t_2 \gg t_3 \gg t_4 > 0\}$ and
  $\sigma' \in \{t_2 \gg t_1 \gg t_3 \gg t_4 > 0\}$ be generic
  cocharacters in these chambers. We compute
  $\hat \se_{\sA \times \bC^\times}$ of the CY specialization of
  \[ \sV_{1/2}^{\sigma,a} - \sV_{1/2}^{\sigma',a} = \left((\sF_\fin^\vee \sF_\fin)_{\substack{>_\sigma\\<_{\sigma'}}} - (\sF_\fin^\vee \sF_\fin)_{\substack{<_\sigma\\>_{\sigma'}}}\right)\tau_{bcd}^\vee - \kappa_{abcd}^\vee (\cdots) \]
  where $\cdots$ denotes the dual of the preceding term. We implicitly
  used here that
  $f_{\ge_\sigma,<_{\sigma'}} = f_{>_\sigma,<_{\sigma'}}$ for any
  $f \in \bk_\sT$ by genericity of $\sigma$ and $\sigma'$. Writing
  this as $\sS_{1/2} - \kappa_{1234}^\vee \sS_{1/2}^\vee$, it suffices
  to show $\rank \bar\sS_{1/2}^{\sA\text{-fix}} = 0$ by the same
  reasoning as in the proof of Lemma~\ref{lem:Vhalf-independent-of-a}.

  Let $w = t_1^i t_2^j t_3^k t_4^l$ be a $\sT$-weight appearing in
  $(\sF_\fin^\vee \sF_\fin)_{>_\sigma,<_{\sigma'}}$. It must be that
  $i > 0$ and $j < 0$. In particular, $i - j > 1$, so
  \[ (i, j, k, l) \not\in \{0, 1\}^4 \bmod{(1,1,1,1)} \]
  because the rhs has $|a - b| \le 1$ for any coordinates $a, b$ in
  the $4$-tuple. Hence $(w \tau_{bcd}^\vee)\big|_\sA$ has no
  $\sA$-fixed part. The same argument works for
  $(\sF_\fin^\vee \sF_\fin)_{<_\sigma,>_{\sigma'}}$.
\end{proof}

\subsection{The Nekrasov--Piazzalunga sign}
\label{sec:NP-sign}

\subsubsection{}

This subsection shows that our half $\sV_{1/2}$ subsumes the half
$\sV_{1/2}^{\text{CKM}}$ prescribed by Cao--Kool--Monavari
\cite[Equation (18)]{Cao2022}, to be defined below, which itself
subsumes the original prescription of Nekrasov for $\Hilb(\bC^4)$
\cite{Nekrasov2020}. All these previous halvings require an additional
non-trivial sign contribution at each $\sT$-fixed point, which has an
explicit conjectural formula, namely the Nekrasov--Piazzalunga sign
\begin{equation} \label{eq:NP-sign}
  \NP_4(\xi) \coloneqq \sum_{\substack{\square=(n_1, n_2, n_3, n_4) \in \xi\\n_1=n_2=n_3<n_4}} \left(1 - \#\{\text{legs containing } \square\}\right),
\end{equation}
{\it only for} DT vertices with empty fourth leg. A meaningful
comparison (Proposition~\ref{prop:CKM-vs-our-half}) with our
$\sV_{1/2}$ is therefore only possible in this setting. Hence for this
subsection we will implicitly only consider DT theory, where in
particular $\sF_{\fin,\ge 0} = \sF_\fin$, and the main goal to prove
Proposition~\ref{prop:CKM-vs-our-half}.

\subsubsection{}

\begin{definition}
  In the notation of Definition~\ref{def:canonical-half}, the {\it CKM
    half} of $\sV$ is
  \begin{align*}
    \sV_{1/2}^{\mathrm{CKM},l}
    &\coloneqq \sF_\fin - \sF_\fin^\vee \sF_\fin \tau_{ijk}^\vee \\
    &\qquad -\bigg(\sF_l^\vee \sF_\fin \tau_{ijk}^\vee + \sum_{i \neq l} \left(\sF_i^\vee \sF_\fin - \sF_\fin^\vee \sF_i t_i^{-1}\right) \tau_{jk}^\vee\bigg) \\
    &\qquad +\bigg(\sum_{i < l} \sF_i^\vee \sF_l t_l^{-1} \tau_{jk}^\vee + \sum_{l < j} \sF_l^\vee \sF_j t_j^{-1} \tau_{ik}^\vee + \sum_{\substack{i<j\\i,j \neq l}} \left(\sF_i^\vee \sF_j t_j^{-1} + \sF_j^\vee \sF_i t_i^{-1}\right) \tau_k^\vee\bigg).
  \end{align*}
  Put differently, and perhaps more clearly, the discrepancy with our
  half is that
  \begin{equation} \label{eq:discrepancy-with-CKM}
    \begin{array}{c|c|c}
      \sV_{1/2}^{\mathrm{CKM},l} \text{ contains} & \sV_{1/2}^{\sigma,l} \text{ contains} & \text{for all} \\ \hline
      \sF_\fin^\vee \sF_\fin \tau_{ijk}^\vee & \left((\sF_\fin^\vee \sF_\fin)_{\ge_\sigma} - ((\sF_\fin^\vee \sF_\fin)_{<_\sigma})^\vee t_l^{-1} \right) \tau_{ijk}^\vee \\
      \left(\sF_i^\vee \sF_\fin - \sF_\fin^\vee \sF_i t_i^{-1}\right) \tau_{jk}^\vee & \sF_i^\vee \sF_\fin \tau_{jkl}^\vee & i \neq l \\
      \left(\sF_i^\vee \sF_j t_j^{-1} + \sF_j^\vee \sF_i t_i^{-1}\right) \tau_k^\vee & \sF_i^\vee \sF_j t_j^{-1} \tau_{kl}^\vee & i, j \neq l
    \end{array}
  \end{equation}
  and all other terms remain the same. Unlike our half, note that {\it
    all} parts (except $\sF_\fin$) of the CKM half depend
  asymmetrically on $l$. While this property facilitates dimensional
  reduction along the $l$-th axis to the $3$-fold setting \cite[\S
  2]{Cao2022}, it significantly complicates the analogue of
  Proposition~\ref{lem:Vhalf-independent-of-a} on independence on the
  choice of $l$; see \cite[Theorem 2.14]{Monavari2022}.
\end{definition}

\subsubsection{}

\begin{proof}[Proof of Proposition~\ref{prop:CKM-vs-our-half}.]
  Both $\sV_{1/2}^{\sigma,l}$ and $\sV_{1/2}^{\mathrm{CKM},l}$ are
  halves of $\sV$, so
  \[ \hat \se_\sA\left(\bar\sV_{1/2}^{\mathrm{CKM},l} - \bar\sV_{1/2}^{\sigma,l}\right) = \hat \se_\sA\left(\bar\sS_{1/2} - \bar\sS_{1/2}^\vee\right) = (-1)^{\rank \bar \sS_{1/2}^{\sA\text{-mov}}} \]
  for some $\sS_{1/2} \in \bk_\sT$; from
  \eqref{eq:discrepancy-with-CKM}, we can take
  \begin{align}
    \sS_{1/2}
    &= (\sF_\fin^\vee \sF_\fin)_{>_\sigma} t_l^{-1} \tau_{ijk}^\vee + \sum_{i \neq l} \sF_i^\vee \sF_\fin t_l^{-1} \tau_{jk}^\vee + \sum_{i,j \neq l} \sF_i^\vee \sF_j t_j^{-1} t_l^{-1} \tau_k^\vee \nonumber \\
    &= \sC(\xi_\fin, \xi) t_l^{-1} \tau_{ijk}^\vee + \sum_{i,j \neq l} \sF_i^\vee \sF_j t_j^{-1} t_l^{-1} \tau_k^\vee \label{eq:CKM-vs-our-half-discrepancy}
  \end{align}
  where $\xi_\fin \subset \xi$ is the collection of boxes in $\xi$
  which do not lie in exactly one leg, and, letting $w(\square)$
  denote the $\sT$-weight of a box $\square \in \xi$,
  \[ \sC(\xi_\fin, \xi) \coloneqq \sum_{\substack{\square \in \xi_\fin\\\square' \in \xi}} \frac{w(\square)}{w(\square')} \cdot \begin{cases} (1 - \nlegs(\square)) & w(\square) >_\sigma w(\square') \\ (1 - \nlegs(\square)) \#\text{legs}(\square') & \text{otherwise}. \end{cases} \]
  Here we used that the $l$-th leg of $\xi$ is trivial, so
  $\sF = \sF_\fin + \sum_{i \neq l} \sF_i$.

  Since $\rank \bar \sS_{1/2} = 0$ is even, it is equivalent to
  compute the rank of the $\sA$-fixed part of $\bar \sS_{1/2}$. This
  rank is zero for the terms in the sum in
  \eqref{eq:CKM-vs-our-half-discrepancy} by
  Lemma~\ref{lem:no-sign-from-leg-leg-part}, and the same parity as
  $\NP_4$ for the first term in \eqref{eq:CKM-vs-our-half-discrepancy}
  by Lemma~\ref{lem:NP-sign-from-our-half}.
\end{proof}

\subsubsection{}

\begin{lemma} \label{lem:no-sign-from-leg-leg-part}
  $\rank \Big((\sF_i^\vee \sF_j t_j^{-1} t_l^{-1} \tau_k^\vee)\Big|_\sA\Big)^{\sA\text{-fix}} = 0$.
\end{lemma}

\begin{proof}
  Without loss of generality, $(i, j, k, l) = (1, 2, 3, 4)$. For boxes
  $(0, a_2, a_3, a_4) \in \sF_1$ and $(b_1, 0, b_3, b_4) \in \sF_2$,
  an $\sA$-fixed term can only arise when
  \[ (b_1, 0, b_3, b_4) - (0, a_2, a_3, a_4) = (b_1, -a_2, b_3-a_3, b_4-a_4) \]
  is congruent to $(0, 1, 0, 1)$ or $(0, 1, 1, 1)$ modulo
  $(1, 1, 1, 1)$. This is impossible because, looking at the first two
  coordinates, $b_1 \ge 0 \ge -a_2$ while $0 < 1$.
\end{proof}

\subsubsection{}

\begin{lemma} \label{lem:NP-sign-from-our-half}
  $\rank\Big((\sC(\xi_\fin, \xi)t_4^{-1} \tau_{123}^\vee)\Big|_\sA\Big)^{\sA\text{-fix}} \equiv \NP_4(\xi) \bmod{2}$.
\end{lemma}

\begin{proof}
  Fix a box $\square \in \xi_\fin$. Suppose $\square' \in \xi$ such
  that the corresponding term in $\sC(\xi_\fin, \xi)$ produces an
  $\sA$-fixed term. We count the contribution of such $\square'$ as
  follows.

  Consider those $\square'$ with $w(\square) >_\sigma w(\square')$.
  These contribute $1 - \nlegs(\square)$. Since $\sigma$ is reverse
  lexicographic order, it must be that
  \[ \square - \square' \in (\{0, 1\}, \{0, 1\}, \{0, 1\}, 1) + \bZ_{\ge 0} \cdot (1, 1, 1, 1). \]
  In other words, we must count the number of boxes in $\xi \cap S_-$
  where
  \[ S_- \coloneqq \square - (\{0, 1\}, \{0, 1\}, \{0, 1\}, 1) - \bZ_{\ge 0} \cdot (1, 1, 1, 1) \]
  is a collection of cubes descending along the $(1,1,1,1)$ axis
  starting at $\square$. This number is even unless there exists
  $\square' \in \xi \cap (\square - \bZ_{\ge 0} \cdot (1,1,1,1))$ such
  that $\square' - \vec e_i \notin \xi$ for any $i = 1, 2, 3$, where
  $\vec e_i$ are standard basis vectors. In this case, the definition
  of a solid partition implies $\square' = (0, 0, 0, d)$ for some
  $d > 0$.

  Now consider those $\square'$ with
  $w(\square) \le_\sigma w(\square')$. These contribute
  $(1 - \nlegs(\square))\nlegs(\square')$. Now we must count the
  number of boxes in $\xi \cap S_+$ where
  \[ S_+ \coloneqq \square + (\{0, 1\}, \{0, 1\}, \{0, 1\}, 0) + \bZ_{\ge 0} \cdot (1, 1, 1, 1), \]
  and this number is even unless there exists
  $\square' \in \xi \cap (\square + \bZ_{\ge 0} \cdot (1,1,1,1))$ such
  that $\square' + \vec e_i \notin \xi$ for any $i = 1, 2, 3$. In this
  case, as legs extend infinitely along their axes and there is no
  fourth leg, $\nlegs(\square') = 0$ so the contribution is zero.

  Adding up all these contributions gives
  \[ \sum_{\substack{\square=(n_1,n_2,n_3,n_4) \in \xi_\fin\\n_1=n_2=n_3<n_4}} (1 - \nlegs(\square)) = \NP_4(\xi), \]
  because by definition all boxes in $\xi \setminus \xi_\fin$ belong
  to exactly one leg.
\end{proof}

\subsection{1-legged case}
\label{sec:1-leg}

\subsubsection{}

Fix a curve $C \coloneqq \bP^1$ and let $X = C \times \bC^3$ with
projection $\pi\colon X \to \bP^1$. Let $x_1, x_2, x_3$ be the
coordinates on $\bC^3$, and $x_4$ be the coordinate on $C$.

\begin{lemma} \label{lem:PT-QM-correspondence}
  There is an isomorphism
  \begin{align}
    \Phi\colon \PT_{d,n-d}(X) &\xrightarrow{\sim} \QM_n(\Hilb(\bC^3,d)) \label{eq:PT-QM-correspondence} \\
    [\cO_X \xrightarrow{s} \cF] &\mapsto [\cO_C \xrightarrow{\pi_*s} \pi_*\cF] \nonumber
  \end{align}
  of moduli spaces with obstruction theory, \footnote{For simplicity,
    we check the agreement of obstruction theories only in the
    K-group, which is sufficient for our purposes, but certainly they
    should agree in $D^b\mathscr{C}\!{\it oh}_\sT$ as well.} where
  $\QM_n$ denotes the moduli space of stable quasimaps
  \cite{Ciocan-Fontanine2014} of degree $n$ from the fixed source
  curve $C$.
\end{lemma}

We view $\Hilb(\bC^3, d)$ as the open locus, in the framed quiver
representation moduli stack
\begin{equation} \label{eq:commuting-three-loop-quiver}
  \left\{\begin{tikzcd}
      \bC \ar{r}{i} & V \ar[loop,in=50,out=90,looseness=6, "B_1"] \ar[loop,in=-20,out=20,looseness=6, "B_2"] \ar[loop,in=-90,out=-50,looseness=6, "B_3"]
    \end{tikzcd} : \begin{array}{c} {[B_1, B_2]} = 0\\ {[B_2,B_3]} = 0\\ {[B_3,B_1]} = 0\\\dim V = d\end{array} \right\}\Big/\GL(V),
\end{equation}
where $i$ is injective. By definition, a quasimap
$f\colon C \dashrightarrow \Hilb(\bC^3, d)$ is a map from $C$ to this
moduli stack, i.e. it is the same kind of quiver representation but
valued in vector bundles on $C$ instead of vector spaces. So we
sometimes denote the quasimap by the vector bundle $\cV$ which is the
analogue of $V$ above. Stability means that $f$ generically lands in
the open locus $\Hilb(\bC^3, d)$.

\begin{proof}
  This result was first observed in one lower dimension \cite[Exercise
  4.3.22]{Okounkov2017}, but the proof is the same so we only briefly
  sketch it. The PT data $\cO_C \xrightarrow{\pi_*s} \pi_*\cF$ is
  exactly a sheafy representation of the quiver above with
  multiplication by $x_i$ on $\pi_*\cF$ as the map $B_i$. Purity of
  $\pi_*\cF$ is clear, and the stability condition for a quasimap $f$
  is equivalent to $\dim \coker \pi_*s = 0$, so indeed PT and quasimap
  stabilities agree. The PT and quasimap obstruction theories are
  identified using
  \[ \Ext_X^*(I, I) = H^*(X, \cExt_X^*(I, I)) = H^*(C, \cExt_\pi^*(I, I)) \in \bk_\sT \]
  where the first equality is the local-to-global Ext spectral
  sequence and the second is the (degenerate) Serre spectral sequence.
  Letting $\cV \coloneqq \pi_*\cF$,
  \begin{equation} \label{eq:quasimap-obstruction-theory}
    \begin{aligned}
      \cExt_\pi^*(I, I)_\perp
      &\coloneqq \cExt_\pi^*(\cO, \cO) - \cExt_\pi^*(I, I) \\
      &= \cV - \cV^\vee \kappa^\vee_{123} - \cV^\vee \cV \tau^\vee_{123} \in K_\sT(C)
    \end{aligned}
  \end{equation}
  is a sheafy version of $T^\vir \Hilb(\bC^3)$, computed in exactly
  the same way. Its cohomology over $C$ is therefore
  $\cT^\vir\QM(\Hilb(\bC^3))$.
\end{proof}

\subsubsection{}

\begin{remark}
  In the $3$-dimensional setting, \eqref{eq:PT-QM-correspondence} is
  the trivial $m=0$ instance of a more general isomorphism
  \begin{equation} \label{eq:BS-QM}
    \BS_{\beta,n-\sum_i\beta_i}(A_m \times \bP^1) \xrightarrow{\sim} \QM_n\left(\Hilb(A_m, \beta)\right)
  \end{equation}
  for $A_m$ surfaces \cite{Liu2021}, where $\BS_{\beta,n}(-)$ is the
  moduli space of Bryan--Steinberg stable pairs on a crepant
  resolution \cite{Bryan2016}. This isomorphism arises via the derived
  McKay correspondence for surface singularities. It is therefore
  natural to expect an analogous isomorphism for $3$-fold
  singularities whose crepant resolutions have isolated fixed points,
  generalizing \eqref{eq:PT-QM-correspondence}.

  Note that, even more generally, \eqref{eq:BS-QM} continues to hold
  for non-trivial $A_m$-fibrations over $\bP^1$, using appropriately
  {\it twisted} quasimaps on the rhs.
\end{remark}

\subsubsection{}

Let $\bC^4 \subset X$ be the toric chart around $0 \in C$. We now
consider the PT vertex picture. For this subsection only, for clarity,
let $q \coloneqq t_4$ and write
$\sT = (\bC^\times)^3 \times \bC_q^\times$.

Let $I$ be a $\sT$-fixed PT stable pair with leg $\pi^4 = \pi$. By
Lemma~\ref{lem:PT-QM-correspondence}, its restriction to $\bC^4$ is
equivalently a $\sT$-fixed stable quasimap $\cV$ which has a
$\sT$-equivariant isomorphism
\[ \cV|_\infty = \pi \in \Hilb(\bC^3), \]
where $\bC_q^\times$ acts trivially on this
$\bC^3 = \pi^{-1}(\infty)$. The PT vertex term $\sV$ from
\eqref{eq:vertex-term} is therefore identified with
\begin{align*}
  \sV
  &= H^*\left(C, \cExt_\pi^*(I, I)_\perp\right) - T_\pi^\vir\Hilb(\bC^3) \\
  &= H^*\left(C, \cExt_\pi^*(I, I)_\perp \otimes \cO_C(-[\infty])\right)
\end{align*}
where the second equality follows from equivariant localization on $C$
and the restriction of \eqref{eq:quasimap-obstruction-theory} to
$\infty \in C$. We disregard the redundant $\bC^\times$-equivariance
of $\sV$ because $\sA$-fixed and $\sT$-fixed PT loci are equal in this
$1$-leg setting, but the remaining equivariance means it is important
to distinguish between $\cO_C(-[\infty])$ and
$\cO_C(-[0]) = q \otimes \cO_C(-[\infty])$.

\subsubsection{}

Definition~\ref{def:canonical-half} of the canonical half $\sV_{1/2}$
of $\sV$ was originally motivated by the following calculation. Recall
that an overline denotes specialization to the CY torus
$\sA \subset \sT$.

\begin{proposition} \label{prop:canonical-half-1-leg-as-QM-Hilb-C3}
  Let
  \[ \sV'_{1/2} \coloneqq H^*\left(C, (\cV - (\cV^\vee\cV)_{<_{\deg}} \tau_{123}^\vee) \otimes \cO_C(-[\infty])\right) \]
  where the subscript $<_{\deg}$ means to take only terms which are
  line bundles of negative degree. Then
  \[ \hat \se_\sA(\bar\sV'_{1/2}) = \hat \se_\sA(\bar\sV_{1/2}). \]
\end{proposition}

Note that $\cV - (\cV^\vee \cV)_{<_{\deg}} \tau_{123}^\vee$ is almost
a half of $\cExt_\pi^*(I, I)_\perp$. The discrepancy consists of
degree-$0$ bundles which contribute nothing to $\sV_{1/2}'$, so
$\sV_{1/2}'$ is indeed a half of $\sV$.

\begin{proof}
  In this $1$-leg setting,
  \begin{equation} \label{eq:canonical-half-1-leg}
    \sV_{1/2} = \sF_\fin - \left((\sF_\fin^\vee \sF_\fin)_{\ge_\sigma} - ((\sF_\fin^\vee \sF_\fin)_{<_\sigma})^\vee q^{-1} - \sF_\fin^\vee \sF_4 q^{-1} \right) \tau_{123}^\vee
  \end{equation}
  where $\sigma$ is reverse lexicographic ordering and $\sF_4$ is the
  $(\bC^\times)^3$-character of $\pi$. Write $w \in \pi$ to mean
  $w \in \sF_4$ and let
  \[ \cV = \sum_{w \in \pi} \cO_C(d_w[0]) \otimes w \in K_\sT(C) \]
  Stability ensures that all $d_w \ge 0$. It is convenient to
  introduce the $q$-integer
  \[ [d] \coloneqq \sum_{i=1}^d q^{-i} = H^0(\cO_C(d[0]-[\infty])), \]
  after which it is easy to match the linear terms
  $\sF_\fin = \sum_{w \in \pi} w[d_w] = H^*(\cV \otimes \cO_C(-[\infty]))$.
  For the quadratic terms, for $w \neq 1$ the identity
  \[ q (w [a]^\vee [b])_{\ge_\sigma} - (w [a]^\vee[b])_{>_\sigma} = w[a]^\vee - w[a-b]^\vee + w\min(a,b) \cdot \begin{cases} 1 & w >_\sigma 1 \\ q & \text{otherwise},\end{cases} \]
  for $a, b \ge 0$, implies that the quadratic terms in the brackets in
  \eqref{eq:canonical-half-1-leg} equal
  \[ -q^{-1} \sum_{w,w' \in \pi} \left(\frac{w}{w'} [d_{w'} - d_w]^\vee + \frac{w'}{w} [d_w - d_{w'}]^\vee\right) + \sS \]
  where $\sS$ denotes a sum of terms $wq^{-1} + w^{-1}$ for weights $w
  >_\sigma 1$. In particular, $\bar\sS \tau_{123}^\vee = \bar\sS_{1/2}
  + \bar\sS_{1/2}^\vee$ with $\rank (\bar\sS_{1/2})^{\sA\text{-fix}} =
  0$, and therefore the term $\sS$ contributes only a trivial sign
  $+1$ and can be ignored. We conclude by using that
  \[ H^*\left(\cO_C(-d[0])_{<_{\deg}} \otimes \cO_C(-[\infty])\right) = -q^{-1} [d]^\vee. \qedhere \]
\end{proof}

\subsubsection{}

The choice of half in
Proposition~\ref{prop:canonical-half-1-leg-as-QM-Hilb-C3} has the
following geometric significance. Recall that
\begin{equation} \label{eq:Hilb-C3-as-critical-locus}
  \Hilb(\bC^3) = \tr \left([B_1, B_2] B_3\right) \subset \Hilb(\Free_3)
\end{equation}
is a critical locus in the Hilbert scheme of the free algebra on three
variables, i.e. the stable locus of
\eqref{eq:commuting-three-loop-quiver} without the three commutativity
relations. Since $\Hilb(\Free_3)$ is smooth, this is a global
Behrend--Fantechi-style presentation which gives
$\cT^\vir \Hilb(\bC^3) = \cT \Hilb(\Free_3) - \kappa_{123}^\vee \otimes \cT_{\Hilb(\Free_3)}^\vee$.
One might then hope that the induced
\[ \QM(\Hilb(\bC^3)) \subset \QM(\Hilb(\Free_3)) \]
is analogously an Oh--Thomas-style presentation --- the zero locus of
a section of an isotropic sub-bundle $\Lambda$ in some
$\SO(2m, \bC)$-bundle on a smooth ambient space $M$ --- which gives
$\cT^\vir_{1/2}\QM(\Hilb(\bC^3)) = \cT_M - \Lambda$, but
$\QM(\Hilb(\Free_3))$ is not smooth because
\[ \cT^\vir \QM(\Hilb(\Free_3)) = H^*\left(C, \cV + \cV^\vee \cV \otimes (t_1 + t_2 + t_3 - 1)^\vee\right) \in \bk_\sT \]
evidently can have negative terms coming from
$H^1\big((\cV^\vee \cV)_{<_{\deg}}\big)$. The half of
$\cT^\vir\QM(\Hilb(\bC^3))$ with no negative terms is given by
\[ H^*\left(C,
    \begin{array}{r} \cV + (\cV^\vee\cV)_{\ge_{\deg}} \otimes (t_1 + t_2 + t_3 + t_1t_2t_3)^\vee \\
      -\, (\cV^\vee\cV)_{<_{\deg}} \otimes (1 + t_1t_2 + t_1t_3 + t_2t_3)^\vee
    \end{array}\right). \]
Presumably the term
$(\cV^\vee\cV)_{\ge_{\deg}} \otimes (t_1 + t_2 + t_3 + t_1t_2t_3)^\vee$
should be removed and its dual added as part of the isotropic
sub-bundle of obstructions, since the dual carries the commutativity
relations $[B_a, B_b] = 0$ (like its non-sheafy version did in
\eqref{eq:Hilb-C3-as-critical-locus}) which must still be imposed.
Making this change, we arrive at
$\cV - (\cV^\vee \cV)_{<_{\deg}} \tau_{123}^\vee$.

It is unclear whether $\QM(\Hilb(\bC^3))$ actually admits a global
Oh--Thomas-style presentation, especially one yielding this half of
$\cT^\vir\QM(\Hilb(\bC^3))$.

\section{Explicit computation}
\label{sec:explicit-computation}

\subsection{Localizable fixed loci}
\label{sec:localizable-fixed-loci}

\subsubsection{}

Explicit computation of the PT vertex, using our main
Conjecture~\ref{conj:DT-PT-Ovir}, requires integration over
$\sT$-fixed PT loci of the form $\Gr_{\vec e}(\bar M)$, which can be
arbitrarily complicated, e.g. by
Example~\ref{ex:PT-full-fixed-locus-Quot-scheme}. In general, we
expect the only systematic method to compute such integrals is by
integration over the $\theta$-stable locus
$W^{\theta\text{-st}}/\GL(\vec e)$ in which $\Gr_{\vec e}(\bar M)$ is
a zero locus (see \eqref{eq:quiver-grassmannian-GIT-presentation}),
via $\GL(\vec e)$-equivariant \v Cech cohomology on
$W^{\theta\text{-st}}$ or otherwise. This is a fairly involved
procedure.

In special cases, particularly in low degree, this procedure can be
replaced by a simpler one using a better geometric understanding of
$\Gr_{\vec e}(\bar M)$, which is the main goal of this section. Such a
simplification was important for the computer verification of
Proposition~\ref{prop:DT-PT-check}.

\subsubsection{}

If $\Gr_{\vec e}(\bar M)$ admits an action by a torus
$\sR = (\bC^\times)^r$, then the perfect obstruction theory
\eqref{eq:quiver-grassmannian-obstruction-theory} defining
$\hat\cO^{\vir,\mathrm{BF}}$ can be made $\sR$-equivariant as well. By
$\sR$-equivariant virtual localization \cite{Graber1999},
\begin{equation} \label{eq:localizable-fixed-loci}
  \hat\cO^{\vir,\mathrm{BF}}_{\Gr_{\vec e}(\bar M)} = \iota_* \frac{\hat\cO^{\vir,\mathrm{BF}}_{\Gr_{\vec e}(\bar M)^\sR}}{\hat \se_\sR(\cN^{\vir,\mathrm{BF}}_\iota)} \in K_\sR(\Gr_{\vec e}(\bar M))
\end{equation}
where $\iota$ is the inclusion of the $\sR$-fixed locus, and the rhs
lives in non-localized K-theory by properness of $\Gr(\bar M)$, even
though its pieces exist only in localized K-theory. Note that this
does not require any smoothness assumption on $\Gr_{\vec e}(\bar M)$.

\begin{definition}
  A pair $(\bar M, \vec e)$ is {\it localizable} if
  $\Gr_{\vec e}(\bar M)$ has an $\sR$-action with isolated fixed
  points, on which the $\sR$-fixed part of the obstruction theory is
  trivial.
\end{definition}

For localizable $(\bar M, \vec e)$, integration over
$\Gr_{\vec e}(\bar M)$ can be performed explicitly using
\eqref{eq:localizable-fixed-loci} because
$\hat\cO^{\vir,\mathrm{BF}} = \cO$ on the $\sR$-fixed locus. In this
subsection, we identify some localizable $(\bar M, \vec e)$.

\subsubsection{}

\begin{proposition} \label{prop:full-torus-action-implies-localizable}
  Suppose that $\bar M$ admits a \emph{full $\sR$-action}, meaning
  that:
  \begin{enumerate}
  \item $\sR$ acts linearly on each $\bar M_v$ such that all maps
    $\bar M_{v \to w}$ are $\sR$-equivariant;
  \item the $\sR$-weights of each $\bar M_v$ are non-trivial and
    distinct from each other.
  \end{enumerate}
  Then $(\bar M, \vec e)$ is localizable for any $\vec e$.
\end{proposition}

\begin{proof}
  The conditions on the $\sR$-action imply that the smooth variety
  $\prod_v \Gr_{e_v}(\bar M_v)$ already has isolated fixed points,
  hence so must the subscheme $\Gr_{\vec e}(\bar M)$. It remains to
  compute the $\sR$-fixed part of the obstruction theory
  \[ \sum_v \left(\Hom(\bar N_v, \bar M_v) - \Hom(\bar N_v, \bar N_v)\right) - \sum_{v \to w} \left(\Hom(\bar N_v, \bar M_w) - \Hom(\bar N_v, \bar N_w)\right) \in K_\sR(\pt). \]
  The first sum is $\oplus_v T \Gr_{e_v}(\bar M_v)$ and therefore has
  no fixed part. So consider the terms of the second sum. We will show
  that
  \[ \Hom(\bar N_v, \bar M_w)^{\sR\text{-fix}} = \Hom(\bar N_v, \bar N_w)^{\sR\text{-fix}} \]
  for any $v \to w$. Terms on the left (resp. right) come from
  non-zero elements $x \in \bar N_v$ and $y \in \bar M_w$ (resp.
  $\bar N_w$) of the same $\sR$-weight. We establish a bijection
  between these terms. Take $y \in \bar M_w$ of the same $\sR$-weight
  as $x$. One checks easily from its definition in
  \S\ref{sec:PT-full-quiver} that all maps in $\bar M$ are either
  injective or surjective, hence
  \[ \bar M_{v \to w}(\bC x) = \bC y \]
  by condition (ii) above. Since $\bar N$ is a quiver
  sub-representation, it follows that $y \in \bar N_w$. Conversely,
  $y \in \bar N_w$ clearly means $y \in \bar M_w$ as well.
\end{proof}

\subsubsection{}

\begin{example} \label{ex:quot-loci-are-localizable}
  All the fixed loci arising from the $\Quot$ scheme of
  Example~\ref{ex:PT-full-fixed-locus-Quot-scheme} are localizable,
  because the relevant portion of $\bar M$ consists only of many
  copies of $\bar \bC_{1234}$ all connected by isomorphisms so one can
  take any standard action of $\sR = (\bC^\times)^3$. This works even
  though such fixed loci are arbitrarily singular.
\end{example}

\subsubsection{}
\label{sec:quiver-grassmannian-reductions}

We can perform a sequence of reductions on $(\bar M, \vec e)$ without
changing the quiver Grassmannian or its perfect obstruction theory.
First, remove all nodes $v \in Q$ where $\vec e_w = 0$ for all
$w \le v$. The $w \le v$ condition is important because otherwise
$\bar N_v = 0$ may impose non-trivial relations, as in
Example~\ref{ex:quiver-grassmannian-singular}. Second, let
$\bar N^f \subset \bar N$ be the quiver sub-representation generated
by all $\bar N_v$ where $\vec e_v = \dim \bar M_v$. This ``fixed''
sub-representation is necessarily present in all
$\bar N \in \Gr_{\vec e}(\bar M)$, so
\begin{equation} \label{eq:quiver-grassmannian-reduction}
  \Gr_{\vec e}(\bar M) \cong \Gr_{\vec e - \dim \bar N^f}(\bar M/\bar N^f)
\end{equation}
as moduli schemes with perfect obstruction theory
\eqref{eq:quiver-grassmannian-obstruction-theory}. Note that if the
linear map from node $v$ to node $w$ is surjective in $\bar M$, then
it is still surjective in $\bar M/\bar N^f$.

Let
$\bar M_{n,i} \coloneqq \{\bar M_v : \dim \bar M_v = n, \; \vec e_v = i\}$.
After these two reductions, it is straightforward to check that
$(\bar M, \vec e)$ must be of the form
\begin{equation} \label{eq:reduced-quiver-maps}
  \begin{tikzcd}
    \bar M_{3,1} \ar[dashed]{r} \ar{d} \ar{dr} \ar{drr} & \bar M_{2,0} \ar[dotted]{dr} \ar[dotted]{d} \\
    \bar M_{3,2} \ar[dashed]{r} \ar[bend right,dashed]{rr} & \bar M_{2,1} \ar[dashed]{r} & \bar M_{1,0}
  \end{tikzcd}
\end{equation}
where arrows denote the direction of all quiver maps when they exist
and all quiver maps within each $\bar M_{n,i}$ are isomorphisms.
Dotted arrows can obviously be removed. The final reduction is to
remove dashed arrows: they are formed by quiver sub-representations
\[ \begin{tikzcd}
    \bar M_v \ar[twoheadrightarrow]{r}{\pi} & \bar M_w \\
    \bC^{i+\dim \ker \pi} \ar{r}{f} \ar[hookrightarrow]{u} & \bC^i \ar[hookrightarrow]{u}
  \end{tikzcd} \]
containing a ``fixed'' sub-representation generated by $\ker\pi$.
which may be removed like in \eqref{eq:quiver-grassmannian-reduction}.
After this final reduction, the nodes $\bar M_{2,0}$ may be removed
altogether.

\begin{definition}
  After this sequence of reductions, we say the resulting pair
  $(\bar M, \vec e)$ is {\it in reduced form}. The quiver
  representation $\bar M$ consists only of the solid arrows in
  \eqref{eq:reduced-quiver-maps} and the nodes linked by them.
\end{definition}

\subsubsection{}

Now we characterize whether $(\bar M, \vec e)$ in reduced form admits
a full $\sR$-action. Then
Proposition~\ref{prop:full-torus-action-implies-localizable} applies
equally well to this $\bar M$ because it uses only the property that
all maps in $\bar M$ are either injective or surjective, which is
still true.

\begin{proposition} \label{prop:full-torus-action-on-reduced-quiver}
  Let $(\bar M, \vec e)$ be in reduced form. Then $\bar M$ admits a
  full $\sR$-action iff there exists a basis for each vector space in
  $\bar M_{3,1}$ such that the subspaces in
  \[ \ker(\bar M_{3,1} \to \bar M_{2,1}), \; \ker(\bar M_{3,1} \to \bar M_{1,0}) \]
  are generated by these basis vectors.
\end{proposition}

\begin{proof}
  Using the discussion of Example~\ref{ex:quot-loci-are-localizable},
  and that $\bar M_{3,2}$ admits no outgoing maps, the maps in
  $\bar M_{3,1} \to \bar M_{3,2}$ do not affect whether $\bar M$
  admits a full $\sR$-action, so they can be ignored. The desired
  bases will correspond to $\sR$-weight spaces in $\bar M_{3,1}$.

  Suppose we are given a basis for each vector space in
  $\bar M_{3,1}$. Let $n$ be the number of nodes in $\bar M_{3,1}$ and
  let $\tilde \sR \coloneqq (\bC^\times)^{3n}$ act by scaling each
  basis vector with a distinct weight. There are two obstructions to
  extending this $\tilde \sR$ action on $\bar M_{3,1}$ to a torus
  action on all of $\bar M$.
  \begin{enumerate}
  \item The maps in $\bar M_{3,1} \to \bar M_{2,1}$ and
    $\bar M_{3,1} \to \bar M_{1,0}$ must be made equivariant. Let $f$
    be such a map. Since $f$ is surjective, $f$ may be made
    $\tilde\sR$-equivariant iff $\ker f$ is generated by
    $\tilde\sR$-weight spaces.
  \item If
    $\bar M_{v_1} \xrightarrow{f_1} \bar M_w \xleftarrow{f_2} \bar M_{v_2}$
    are two maps in $\bar M_{3,1} \to \bar M_{2,1}$, step (i) puts two
    incompatible $\tilde\sR$ actions on $\bar M_w$. These two actions
    can be made equal by passing to the appropriate sub-torus of
    $\tilde \sR$.
  \end{enumerate}
  Let $\sR \subset \tilde \sR$ be the resulting sub-torus after
  executing step (ii) for all $\bar M_w \in \bar M_{2,1}$. By step
  (i), all quiver maps are $\sR$-equivariant, so $\sR$ acts on
  $\bar M$. This action is full because step (ii) preserves the
  property that the torus weights in each $\bar M_v \in \bar M_{3,1}$
  are non-trivial and distinct from each other. (It only identifies
  weights in different $\bar M_{v_1}, \bar M_{v_2} \in \bar M_{3,1}$.)
\end{proof}

\subsubsection{}

\begin{remark}
  The characterization of
  Proposition~\ref{prop:full-torus-action-on-reduced-quiver} is a
  combinatorial condition. In fact, tracing through the reductions of
  \S\ref{sec:quiver-grassmannian-reductions}, one can check that each
  $\bar M_v$ is still of the form $\bar \bC_S$ for some
  $S \subset \{1,2,3,4\}$ and kernels of all maps in $\bar M$ are
  subspaces generated by (images of) standard basis vectors $e_i$.
  Hence the desired bases are formed by (images of) three, out of
  four, {\it standard} basis vectors at each
  $\bar M_v \in \bar M_{3,1}$.
\end{remark}

\subsection{Unobstructed fixed loci}
\label{sec:unobstructed-fixed-loci}

\subsubsection{}

If $\Gr_{\vec e}(\bar M)$ is smooth of dimension equal to the rank of
the perfect obstruction theory
\eqref{eq:quiver-grassmannian-obstruction-theory} on it, then
\[ \hat\cO^{\vir,\mathrm{BF}} = \hat\cO \coloneqq \cO \otimes \det(\cT)^{-1/2} \]
can be explicitly identified. Following standard terminology, we say
that such a $(\bar M, \vec e)$ is {\it unobstructed}.

Being unobstructed is independent of being localizable, though a
general PT fixed component in sufficiently high degree has neither
property. In this subsection, we identify some unobstructed
$(\bar M, \vec e)$. The following
Example~\ref{ex:simplest-non-localizable-locus} is the basic case.

\subsubsection{}

\begin{example} \label{ex:simplest-non-localizable-locus}
  Take the legs
  \[ \pi^1 = \pi^2 = \pi^3 = \pi^4 = \raisebox{-.35\height}{\includegraphics[scale=1.5]{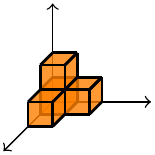}} \]
  so that the relevant portion of $\bar M$ is
  \[ \begin{tikzcd}[column sep=1em, row sep=1em]
      & \bar \bC_{234} \ar{dl} && \bar \bC_{134} \ar{dr} \\
      \bar \bC_{24} && \bar \bC_{1234} \ar{ur} \ar{ul} \ar{dr} \ar{dl} && \bar \bC_{13} \\
      & \bar \bC_{124} \ar{ul} && \bar \bC_{123} \ar{ur}
    \end{tikzcd} \]
  Consider $\Gr_{\vec e}(\bar M)$ where $\vec e$ is $1$ at all nodes
  drawn above in $\bar M$ and $0$ elsewhere. The
  $\bC \subset \bar \bC_{13}$ and $\bC \subset \bar \bC_{24}$ can be
  ignored because they do not constrain anything.

  There is a $\bP^2$ worth of freedom for the central
  $\bC \subset \bar \bC_{1234}$, which fully determines all four
  $\bC \subset \bar \bC_S$ for $|S| = 3$ {\it unless} it is in the
  kernel of one of the four maps $\bar \bC_{1234} \to \bar \bC_S$.
  These kernels are four distinct points $p_i \in \bP^2$ and one
  checks easily that
  \[ \Gr_{\vec e}(\bar M) = \Bl_{p_1,p_2,p_3,p_4}\bP^2. \]
  This is clearly not localizable; in the language of
  Proposition~\ref{prop:full-torus-action-on-reduced-quiver},
  $\ker(\bar M_{3,1} \to \bar M_{2,1})$ consists of the four lines
  $\bC e_i \subset \bar \bC_{1234}$ which cannot all be basis vectors.
  However, $\Bl_{p_1,p_2,p_3,p_4}\bP^2$ is a smooth surface, and the
  perfect obstruction theory has rank two.
  
  Let $\scN_c$ and $(\scN_i)_{i=1}^4$ be the universal line bundles of
  $\bC \subset \bar \bC_{1234}$ and the four $\bC \subset \bar\bC_S$
  respectively. They correspond to $\cO(-H)$ and
  $(\cO(E_i - H))_{i=1}^4$ where $H$ is the hyperplane and $E_i$ are
  exceptional divisors. Explicit calculation gives
  \[ \Big(\bar\sV_{1/2}\Big|_{\Gr_{\vec e}(\bar M)}\Big)^{\sA\text{-fix}} = \left(-5\scN_c - 5\right) + \sum_{i=1}^4 \left(\scN_c^\vee \otimes \scN_i + 2\scN_i\right) \]
  while
  $\cT\Gr_{\vec e}(\bar M) = (-5\scN_c^\vee - 5) + \sum_{i=1}^4 (\scN_c^\vee \otimes \scN_i + 2\scN_i^\vee)$.
  It is a non-trivial check of Conjecture~\ref{conj:DT-PT-Ovir} that
  these two expressions agree up to some duals; the sign of this
  component is therefore $-1$.
\end{example}

\subsubsection{}

Suppose $(\bar M, \vec e)$ is in reduced form, i.e. all the reductions
of \S\ref{sec:quiver-grassmannian-reductions} were performed. Suppose
further that both $\bar M_{1,0}$ and $\bar M_{3,2}$ are empty, so that
the singularities in Examples~\ref{ex:quiver-grassmannian-singular}
and \ref{ex:PT-full-fixed-locus-Quot-scheme} cannot occur. In this
setting, we can characterize unobstructedness.

\begin{proposition}
  Let $(\bar M, \vec e)$ be of the form
  $\bar M_{3,1} \to \bar M_{2,1}$. Then $(\bar M, \vec e)$ is
  unobstructed iff every pair of nodes in $\bar M_{3,1}$ has at most
  two common neighbors in $\bar M_{2,1}$.
\end{proposition}

\begin{proof}
  Consider the subgraph given by one node $\bar M_v \in \bar M_{3,1}$
  and its $m$ neighbors $(\bar M_{w_i})_i$ in $\bar M_{2,1}$. Like in
  Example~\ref{ex:simplest-non-localizable-locus}, one checks easily
  that the quiver Grassmannian associated to this subgraph is an
  unobstructed $\Bl_m \bP^2$. (We write $\Bl_m \bP^2$ instead of
  $\Bl_{p_1,\ldots,p_m}\bP^2$ for short.) The forgetful map which
  remembers only the subspace $\bC \subset \bar M_{w_i}$
  can be written as the composition
  \begin{equation} \label{eq:blowup-of-P2-projection-to-P1}
    \pi_{v \to w_i}\colon \Bl_m\bP^2 \to \Bl_{[\ker \bar M_{v \to w_i}]}\bP^2 \cong \bP(\cO_{\bP^1} \oplus \cO_{\bP^1}(-1)) \to \bP^1
  \end{equation}
  of the obvious maps. The key observation, for later, is that
  $\prod_{i=1}^k \pi_{v \to w_i}\colon \Bl_m\bP^2 \to (\bP^1)^k$ is a
  submersion iff $k \le 2$.

  The quiver Grassmannian for $(\bar M, \vec e)$, as a scheme,
  therefore has the description
  \begin{equation} \label{eq:unrestricted-quiver-grassmannian-as-fiber-product}
    \Gr_{\vec e}(\bar M) = \bigcap_{w\colon \bar M_w \in \bar M_{2,1}} \{\pi_{v_1 \to w} = \cdots = \pi_{v_k \to w}\} \subset \prod_{v\colon \bar M_v \in \bar M_{3,1}}\Bl_{n_v}\bP^2
  \end{equation}
  where, for given $w$, the $v_i$ range over all incident edges
  $(\bar M_{v_i \to w})_i$. It is an easy computation that the rank of
  the perfect obstruction theory on $\Gr_{\vec e}(\bar M)$ is the
  expected dimension, i.e. the dimension of the ambient space minus
  the number of equations. If every pair of nodes in $\bar M_{3,1}$
  has at most two common neighbors, then
  \eqref{eq:unrestricted-quiver-grassmannian-as-fiber-product} is an
  iterated sequence of fiber products of submersions and hence is
  smooth of expected dimension. In the presence of three or more
  common neighbors, there is the phenomenon that
  \[ \Bl_3 \bP^2 \times_{\pi_1 \times \pi_2 \times \pi_3} \Bl_3 \bP^2 = \Bl_{p_1,p_2,p_3}\bP^2 \]
  is a smooth variety with higher than expected dimension, ruining
  unobstructedness. Here the $\pi_i$ are the three possible maps
  \eqref{eq:blowup-of-P2-projection-to-P1} from $\Bl_3\bP^2$.
\end{proof}

\phantomsection
\addcontentsline{toc}{section}{References}

\begin{small}
\bibliographystyle{alpha}
\bibliography{vertex}
\end{small}

\end{document}